\documentclass{sn-jnl}
\usepackage[T1]{fontenc}
\usepackage[utf8]{inputenc}
\usepackage[english]{babel}
\usepackage{latexsym}
\usepackage{url}    

\usepackage{float}

\usepackage{tikz}
\usepackage{xcolor}

\usepackage{cite}
\usepackage{amsmath}
\usepackage{amssymb}
\usepackage{amsfonts}
\usepackage{graphicx}
\usepackage{float}
\usepackage{subcaption}
\usepackage{upgreek}
\usepackage{ mathrsfs }
\usepackage{ dsfont }
\usepackage{amsthm}
\usepackage{graphicx}
\usepackage{caption}
\usepackage{booktabs}
\usepackage{color}
\usepackage{framed}
\usepackage{verbatim}
\usepackage{pdflscape,rotating,arydshln}
\usepackage{afterpage}
\usepackage{capt-of}
\usepackage{bm}
\definecolor{shadecolor}{RGB}{248,248,248}

\theoremstyle{plain}
\newtheorem{theorem}{Theorem}[section]
\newtheorem{definition}[theorem]{Definition}
\newtheorem{example}[theorem]{Example}
\newtheorem{remark}[theorem]{Remark}
\newtheorem{corollary}[theorem]{Corollary}
\newtheorem{proposition}[theorem]{Proposition}
\newtheorem{lemma}[theorem]{Lemma}

\newcommand{\N}{\mathbb{N}}

\newcommand{\R}{\mathbb{R}}

\usepackage{hyperref}

\numberwithin{equation}{section}

\newcommand{\mfk}{\mathfrak{f}}

\newcommand{\F}{\mathscr{F}}

\newcommand{\calZ}{\mathcal{Z}}
\newcommand{\scrZ}{\mathscr{Z}}

\newcommand{\Prob}{\mathbb{P}}
\newcommand{\E}{\mathbb{E}}

\providecommand{\abs}[1]{\lvert #1\rvert}

\providecommand{\Enorm}[1]{\lVert #1\rVert_2}

\DeclareMathOperator{\dom}{dom}

\renewcommand{\phi}{\varphi}
\renewcommand{\epsilon}{\varepsilon}

\renewcommand{\rho}{\varrho}
\renewcommand{\P}{\mathbb{P}}

\begin{document}

\title[The Poisson multiplication formula]{\bf The Poisson Multiplication Formula}

\author[1]{\fnm{Lorenzo} \sur{Cristofaro}}\email{lorenzo.cristofaro@uni.lu}

\author[1]{\fnm{Giovanni} \sur{Peccati}}\email{giovanni.peccati@uni.lu}

\affil[1]{\orgdiv{Department of Mathematics}, \orgname{University of
Luxembourg}; \orgaddress{\street{6, Avenue de la Fonte}, \city{Esch-sur-Alzette}, \postcode{4364}, \country{Luxembourg}}}

\abstract{ We establish necessary and sufficient conditions implying that the product of $m\geq 2$ Poisson functionals, living in a finite sum of Wiener chaoses, is square-integrable. Our conditions are expressed in terms of iterated add-one cost operators, and are obtained through the use of a novel family of {\em Poincar\'e inequalities} for almost surely finite random variables, generalizing the recent findings by Trauthwein (2024). When specialized to the case of multiple Wiener-It\^o integrals, our results yield general multiplication formulae on the Poisson space under minimal conditions, naturally expressed in terms of partitions and diagrams. Our work addresses several questions left open in a seminal work by Surgailis (1984), and completes a line of research initiated in D\"obler and Peccati (2018).
}



\maketitle

 \section{Introduction}\label{sec:int}
\subsection{Overview}

The goal of this paper is to establish necessary and sufficient conditions for the {\bf square-integrability} of the product of an arbitrary number of {\bf multiple Wiener-It\^o integrals} with respect to a general {\bf random Poisson measure}, and to deduce an explicit analytical expression for the associated {\bf multiplication formulae}. As demonstrated below, our main results (stated succinctly in Theorem \ref{t:mainintro} and detailed in the subsequent sections) address several open problems raised in the classical work by Surgailis \cite{surgailis84}. They also complete --- by introducing new ideas and techniques of independent interest --- a line of research started by D\"obler and Peccati in \cite{DoblerPeccati18}; see also \cite{Kab76, DTGS23, RW97}. We stress that~\cite{surgailis84} is a cornerstone of modern stochastic analysis, demonstrating how multiplication formulae can be used to establish the {\bf non-hypercontractivity} of the Ornstein--Uhlenbeck semigroup on the Poisson space, a result that paved the way for the development of {\bf modified logarithmic Sobolev inequalities} on configuration spaces and related concentration estimates; see e.g.~\cite{GuSaTh, BachmannPeccati2016, wu2000, bobkovledoux, chafai8, AdamczakPivovarovSimanjuntak2024, GieLast, NPYpams}.

Because of the well-known {\bf Wiener-It\^o chaos expansion} of Poisson functionals (see formula \eqref{chaosdec}, as well as references \cite{Lastsv, LPfock, LPbook, PTbook, privaultbook}), the multiple stochastic integrals studied in this paper constitute the basic building blocks of generic random variables depending on a given Poisson measure. As demonstrated below, their use is frequently simplified by a fundamental formula due to Last and Penrose (see \cite{LPfock} and the forthcoming Theorem \ref{t:lastpenrose}), which connects the chaos expansion of a Poisson functional to the expected value of iterated {\bf add-one cost operators} --- see Section~\ref{ss:basicintro} and formula \eqref{e:addone}.






Since the appearance of the seminal reference \cite{SchulteThaele12}, and due to the geometric nature of the Last-Penrose formula mentioned above, the study of fluctuations of multiple integrals has gained increasing importance in the probabilistic analysis of geometric models based on random point configurations, often in connection with techniques based on Stein's method (see e.g. \cite{DoblerPeccatiAOP, PSTU, PRbook, DVZ18}). This is particularly evident in the context of {\bf $U$-statistics} \cite{ST24,PeccatiThaeleALEA,Thomas2023,LeMinh2023,PianoforteTurin2023,BourguinPeccati2014,DoblerPeccati2017} and more general geometric functionals \cite{Schulte12,KabluchkoRosenThale2024,BaciBonnetThale2022,ReitznerSchulteThale2017,FisslerThale2016,HugLastSchulte2016,Herry2020,HugThaleWeil2015,LPST,LRPEJP,LRPSPA,BourguinCampeseDang2024,CongXia2024,EichelsbacherThale2014,BachmannPeccati2016} (sometimes in non-standard settings \cite{DurMaPec,DecHomology,BetkenHugThale2023,Grygierek2020,AkinwandeReitzner2020,BachmannReitzner2018,BourguinDurastanti2017, STT24}) as well as in the framework of sensitivity analysis for percolation-type models \cite{BPY,LPY}. The kind of multiplication formulae considered in the present paper plays an important role, e.g., in the analysis of {\bf central} and {\bf non-central limit theorems} for elements of Wiener chaoses \cite{DoblerPeccati18, PSTU, PeccatiThaeleALEA, DVZ18, Schulte2016}, and in the derivation of {\bf second-order results for $U$-statistics} \cite{SchulteThaele12, LRRsv, LRPEJP, LRPSPA, ReitznerSchulteThale2017, ST24}.

We will see in Section \ref{sec:PoincInequaIntegrAssump} that one of our main technical tools is an extension of the class of {\bf $p$-Poincar{\'e} inequalities} established by Trauthwein in \cite{trauthwein}. See also \cite{LPS, multitara}.

\smallskip
For the rest of the paper, every random element is defined on a common probability space $(\Omega, \mathscr{F}, \P)$, with $\E$ denoting expectation with respect to $\P$. Given a real-valued mapping {\color{blue} $H(z_1,...,z_K)$} in $K$ variables, we define the {\bf symmetrization} of {$H$} to be the function
\begin{equation}\label{e:sym}
{(z_1,...,z_K)\mapsto {\rm sym}( H) (z_1,...,z_K)  := \frac{1}{K!} \sum_p H(z_{p(1)},..., z_{p(K)}),}
\end{equation}
where the sum runs over all permutations $p$ of the set $[K] := \{1,...,K\}$.
\subsection{The problem}\label{ss:basicintro}

We will now introduce a minimal amount of notation, which will allow us to state and discuss our main contributions. The reader is referred to Section \ref{sec:Prelimi} for an exhaustive presentation, complete with technical details and pointers to the literature.

\smallskip 

We denote by $\eta$ a Poisson measure on a measurable space $(\mathcal{Z},\mathscr{Z})$, with $\sigma$-finite intensity $\mu$. Given a random variable $F = F(\eta)$ and $z\in \mathcal Z$, we write $D_z^+ F = D_z^+ F(\eta)  := F(\eta+\delta_z) - F(\eta)$ to indicate the add-one cost operator of $F$ at the point $z$, {where $\delta_z$ indicates the Dirac mass at $z$}. For $k\geq 1$ and $z_1,...,z_k\in \mathcal Z$, we also define $D^{(k)}_{z_1,...,z_k} F := D^+_{z_1}\cdots D^+_{z_k} F$ (so that $D^+ = D^{(1)}$); finally, we set $D^{(0)} := {\rm Id.}$. For $k=0,1,2,...$ and for a symmetric $f\in L^2(\mu^k)$, we write $I_k(f)$ to denote the multiple Wiener-It\^o integral of order $k$ of $f$ with respect to $\eta$, with the obvious identification $L^2(\mu^0)= \R$ and $I_0(c) = c$. We recall (Wiener-It\^o chaos expansion) that every square-integrable $F = F(\eta)$ can be uniquely written as an infinite series of the type $F = \sum_{k=0}^\infty I_k(f_k)$, with $f_k = k!^{-1}\E[ D^{(k)} F]$ (Last-Penrose formula). For $k\geq 0$, the collection of all Poisson multiple integrals of order $k$ is referred to as the $k$th Wiener chaos associated with $\eta$.

\smallskip 

As anticipated, the goal of the present work is to address the following problem.

\smallskip 

\noindent{\bf Problem A.} {\em Let $m\geq 2$, consider integers $k_1,...,k_m\geq 1$, and, for $i=1,...,m$, let $f_i$ be a symmetric element of $L^2(\mu^{k_i})$, define $F_i := I_{k_i}(f_i)$, and set
$$
\Phi:= \prod_{i=1}^m F_i.
$$
Then:
\begin{itemize}
\item[\rm (i)] Find necessary and sufficient conditions ensuring that 
\begin{equation}\label{e:phiintro}
\Phi \in L^2(\mathbb{P});
\end{equation}
\item[\rm (ii)] When $\Phi\in L^2(\mathbb{P})$, write explicitly the Wiener-It\^o chaos decomposition of $\Phi$ as a sum of multiple integrals.
\end{itemize}
}

\smallskip 

\begin{remark}{\rm

\begin{enumerate}
    \item If $\eta$ is replaced by a Gaussian measure $G$ with intensity $\mu$ (see \cite[Example 2.1.4]{NouPecbook}), the random variables $F_i$ become Gaussian multiple Wiener-It\^o integrals (see \cite[Section 2.7.1]{NouPecbook}), and the corresponding version of {\bf Problem~A} admits a classical solution. Indeed, in this case the {\bf hypercontractivity} of Gaussian Wiener chaoses (see \cite[Theorem 2.7.2]{NouPecbook}), implies that $\Phi \in L^p(\P)$ for all $p\geq 1$, and the resulting chaos expansion can be made explicit by iterating product formulae based on the use of {\bf contractions}, such as the ones presented in \cite[Theorem 2.7.10]{NouPecbook} or \cite[Section 6.4]{PTbook}. We stress that --- according e.g. to \cite{NPYpams, surgailis84} --- Poisson Wiener chaoses are {\it not} hypercontractive, Poisson multiple integrals may not belong to any space $L^p(\P)$, $p>2$, and the solution to {\bf Problem~A} is consequently non trivial.
\item In the case where the kernels $f_i$ are finite linear combinations of indicators of (hyper)rectangles with finite measure, one has that $\Phi$ admits moments of all orders, and the corresponding chaos expansion can be written explicitly by using the formalism of partitions and diagonal sets --- see e.g. \cite[Sections 6.1 and 6.5]{PTbook}. One of the principal contributions of the present work (see e.g. Theorem \ref{t:mainintro}) is an extension of such a result to generic collections of integrands $f_1,...,f_m$.
\end{enumerate}

}
\end{remark}

\subsection{Existing results and main contributions}\label{ss:existintro}

One partial solution to {\bf Problem A} in the case of an arbitrary integer $m\geq 2$ (containing the findings of Kabanov \cite{Kab76} as a special case) appears in Surgailis' seminal work \cite{surgailis84}. In order to state Surgailis' result, we introduce a portion of the partition-based formalism that will be fully developed in Section \ref{ss:partitions}. Given $m\geq 1$ and a vector $(k_1,\dots,k_m)\in \mathbb{Z}_+^m:= \{1,2,\dots\}^m$, we write $K := k_1+\cdots +k_m$, and denote by $\pi^\star$ the partition of $[K]$ given by
\begin{equation}\label{e:blocks}
\pi^\star := \Big\{ \{1,\dots,k_1\}, \{k_1+1,\dots, k_1+k_2\}, \dots, \{k_1+\cdots +k_{m-1}+1,\dots, K\}\Big\}
\end{equation}
(in other words, $\pi^\star$ is obtained by considering consecutive blocks of integers, with sizes $k_1, k_2, \dots, k_m$, respectively). Adopting the notation introduced in \cite[Section 12.2]{LPbook}, we write $\Pi(k_1,\dots,k_m)$ to indicate the class of all partitions $\sigma$ of $[K]$ such that each block $b\in \sigma$ is such that, $| b\cap b^\star |\in \{0,1\}$ for every block $b^\star \in \pi^\star$, that is, every block of $\sigma$ has at most one element in common with each block of $\pi^\star$. We observe that, in the parlance of \cite[p. 48]{PTbook}, the partition $\sigma$ can be identified with a {\bf non-flat diagram}. Given $\sigma\in\Pi(k_1,\dots,k_m) $, one writes $\sigma_{1}$ and $\sigma_{\geq 2}$, respectively, to indicate the collection of all singletons of $\sigma$, and the collection of all blocks of $\sigma$ of size $\geq 2$ (that is, $\sigma_{\geq 2} = \sigma\, \backslash \,  \sigma_1$); observe that $\sigma_1$ and $\sigma_{\geq 2}$ can be empty (but not simultaneously!). Now fix a pair $(\sigma,A)$, where $\sigma\in \Pi(k_1,...,k_m)$ and $A\subseteq \sigma_{\geq 2}$; given symmetric functions $f_1,...,f_m$ as in {\bf Problem A}, we define a new function {$H(\sigma, A; f_1,...,f_m)$} in $|A| + |\sigma_1|$ variables as follows (assuming all integrals are well-defined):
\begin{itemize}
\item[(1)] Consider the function in $K$ variables given by
\begin{equation*}
{ f_1\otimes \cdots \otimes f_m (v_1,\dots,v_K) :=  \prod_{i=1}^m f_i(v_{k_1+\cdots+ k_{i-1}+1},\dots, v_{k_1+\cdots +k_i}), \, (v_1,\dots,v_K) \in \mathcal{Z}^K};
\end{equation*}

\item[(2)] Identify two variables $v_i, v_j$ in the argument of $f_1\otimes \cdots \otimes f_m$ if and only if $i$ and $j$ are in the same block of $ \sigma_{\geq 2}$;
\item[(3)] For each block $b \in \sigma_{\geq 2} \setminus A$, integrate with respect to $\mu$
the variable resulting from the identification of those $v_j$ such that $j\in b$;
\item[(4)] Symmetrize the resulting expression according to~\eqref{e:sym}, and express it as a function of the variables identified by the blocks in $A \cup \sigma_1$, labeled as $z_1, \dots, z_{|A| + |\sigma_1|}$.

\end{itemize}
If $\sigma = \hat{0}$, where $\hat{0}$ stands for the {\bf minimal} partition whose blocks are the singletons of $[K]$ (see e.g. {\rm \cite[Definition 2.2.3]{PTbook}}), then $\sigma = \sigma_1$ and $\sigma_{\geq 2}=  \emptyset$, and consequently $$H(\sigma, \emptyset ; f_1,\cdots,f_m)=H(\hat{0}, \emptyset ; f_1,\cdots,f_m) ={\rm sym} (f_1\otimes \cdots \otimes f_m).$$ {See Examples \ref{ex:2} and \ref{exam:ProductmFunctions} for further illustrations of the previous construction.}
\begin{figure}[!t]
\centering
\begin{tikzpicture}[scale=1]

\def\dx{0.6}
\def\offset{5.2} 

\foreach \i in {1,...,6} {
  \node[circle, fill=black, inner sep=1pt, label=below:\i] at ({\i*\dx},0) {};
}
\draw[thick] ({2.5*\dx}, -0.05) -- ({2.5*\dx}, 0.1);
\draw[thick] ({4.5*\dx}, -0.05) -- ({4.5*\dx}, 0.1);
\draw[red, thick] ({2*\dx},0) .. controls ({2.5*\dx},0.5) .. ({3*\dx},0);
\draw[red, thick] ({4*\dx},0) .. controls ({4.5*\dx},0.3) .. ({5*\dx},0);
\draw[red, thick] ({3*\dx},0) .. controls ({4.5*\dx},0.6) .. ({6*\dx},0);

\foreach \i in {1,...,6} {
  \node[circle, fill=black, inner sep=1pt, label=below:\i] at ({\i*\dx + \offset},0) {};
}
\draw[thick] ({2.5*\dx + \offset}, -0.05) -- ({2.5*\dx + \offset}, 0.1);
\draw[thick] ({4.5*\dx + \offset}, -0.05) -- ({4.5*\dx + \offset}, 0.1);
\draw[red, thick] ({2*\dx + \offset},0) .. controls ({2.5*\dx + \offset},0.5) .. ({3*\dx + \offset},0);
\draw[blue, thick] ({4*\dx + \offset},0) .. controls ({4.5*\dx + \offset},0.3) .. ({5*\dx + \offset},0);
\draw[red, thick] ({3*\dx+\offset},0) .. controls ({4.5*\dx+\offset},0.6) .. ({6*\dx+\offset},0);

\node at ({3.5*\dx}, -0.9) {(a) $A_1 = \sigma_{\geq 2} =\{\{2,3,6\}, \{4,5\}\} $};
\node at ({3.5*\dx + \offset}, -0.9) {(b) $A_2 =\{\{2,3,6\}\} $};

\end{tikzpicture}
\caption{\it Visualization of two pairs $(\sigma,A_i)$, $i=1,2$, where $m=3$, $k_1=k_2=k_3=2$, $\sigma =\{\{1\}, \{2,3,6\}, \{4,5\}\}$ and $A_i\subseteq \sigma_{\geq 2}$ corresponds to the red blocks.}
\label{f:2diag}
\end{figure}
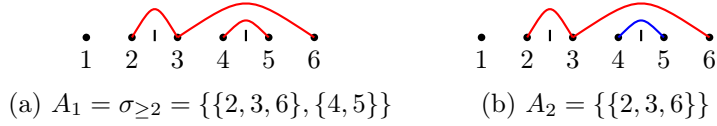

\smallskip 

\begin{remark}\label{r:indy}{\rm Computing the quantity
\begin{equation}\label{e:hgrande}
H(\sigma, A; f_1,\dots,f_m)(z_1,\dots,z_{|A|+|\sigma_1|})
\end{equation}
introduced above, requires partitioning the argument of each mapping \( f_i \), for \( i = 1,\dots,m \), into two subsets of variables \( (\mathbf{z}_i, \mathbf{v}_i) \), such that \( \mathbf{z}_i \subseteq \{z_1,\dots,z_{|A|+|\sigma_1|} \} \), and \( \mathbf{v}_i \) is integrated with respect to an appropriate power of \( \mu \). In what follows, we will refer to the mappings
\[
\mathbf{v}_i \mapsto f_i(\mathbf{z}_i, \mathbf{v}_i), \quad i=1,...,m,
\]
as the {\bf individual kernels} integrated in the definition of \eqref{e:hgrande}. 

}
\end{remark}

\bigskip

\begin{example}\label{ex:2}
{\rm In the setting of {\bf Figure \ref{f:2diag}}, one has that
\begin{eqnarray*}
&& {H(\sigma, A_1; f_1,f_2,f_3)(z_1, z_2, z_3)}   \\
&&=\frac16 \sum_{p} f_1(z_{p(1)}, z_{p(2)})f_2(z_{p(2)}, z_{p(3)})f_3(z_{p(3)}, z_{p(2)}), \quad\mbox{and} \\
&& {H(\sigma, A_2; f_1,f_2,f_3)(z_1, z_2)} \\
&& {= \int_\mathcal{Z} (f_1(z_1,z_2) f_2(z_2,v) f_3(v,z_2) +f_1(z_2,z_1) f_2(z_1,v) f_3(v,z_1))  \mu(dv),}
\end{eqnarray*}
where the first sum runs over all permutations $p$ of the set $\{1,2,3\}$. In this case, there is no kernel integrated in the definition of $H(\sigma, A_1; f_1,f_2,f_3)(z_1, z_2, z_3)$, whereas the individual kernels integrated in the definition of $H(\sigma, A_2; f_1,f_2,f_3)(z_1, z_2)$ correspond to the four mappings
$$
v\mapsto f_2(z_2,v), \, f_3(v,z_2), \, f_2(z_1,v), \, f_3(v,z_1).
$$


}
\end{example}
\medskip

The following result is one of the main findings in \cite{surgailis84}.

\medskip

\begin{theorem}
[\bf Proposition 3.1 in \cite{surgailis84}] \label{pro:SurgaiProp3.1} Let the above notation prevail, and consider the setting of {\bf Problem A}. Assume that, for all $\sigma \in {\Pi}(k_1,\dots,k_m)$ and for all $A \subseteq \sigma_{\geq 2}$,
    \begin{equation} \label{eq:SurgaAssump}
      {H(\sigma, A; |f_1|, \dots, |f_m|) \in L^2(\mu^{|A| + |\sigma_1|}),}
    \end{equation} 
    Then, 
    \begin{equation} \label{eq:SquareIntegrabProductMSI}
        \Phi:=\prod_{i=1}^m I_{k_i}(f_i) \in L^2(\Prob)
    \end{equation}
    and 
    \begin{equation}\label{eq:SurgaiProducDecompo} \Phi \in \bigoplus_{q=0}^{K} C_q \end{equation}
    where $C_q$ is the $q$th Wiener chaos of $\eta$. Moreover, the Wiener-It\^o chaos expansion of $\Phi$, written
    \begin{equation}\label{e:introchaosphi}
    \Phi = \mathbb{E}(\Phi) + \sum_{q=1}^K I_q(h_q),
    \end{equation}
    is such that, for all $q=0,1,...,K$,
    \begin{equation}\label{e:hq}
    h_q =  \sum_{\substack{\sigma \in {\Pi}(k_1,\dots,k_m)\\ A\subseteq \sigma_{\geq 2} \\ |A| +|\sigma_1|=q }} H(\sigma, A; f_1, \dots, f_m) ,
    \end{equation}
    where we have adopted the notation \eqref{e:sym} and $h_0= \E[\Phi]$; in particular, $h_K = {\rm sym}(f_1\otimes\cdots\otimes f_m)$.

\end{theorem}

\medskip

\begin{remark}\label{r:problems}{\rm 
\begin{enumerate}

\item Requirement \eqref{eq:SurgaAssump}
    implies that the quantity {$H(\sigma, A; f_1, \dots, f_m)(z_1,..., z_{|A| +|\sigma_1|})$} is well-defined and finite a.e.-$\mu^{|A|+|\sigma_1|}$. Note that the kernel $h_q$ in \eqref{e:hq} could be equivalently expressed by using the language of {\bf contractions} developed e.g. in \cite{DTGS23}.

\item It is a well-known fact, and a direct consequence of Theorem \ref{t:diagram} below (where the forthcoming relation \eqref{e:dumbo} is expressed in a slightly different notation) that, if $f_i\in L^1(\mu^i)$, $i=1,...,m$, and {$H(\sigma, \emptyset ; f_1, \dots, f_m)$} is well-defined and finite for all $\sigma$ such that $\sigma_1 = \emptyset$ (note that {$ H(\sigma, \emptyset ; f_1, \dots, f_m)$} is, by definition, a constant), then 
\begin{equation}\label{e:dumbo}
    \E[\Phi] =  \sum_{\substack{\sigma \in {\Pi}(k_1,\dots,k_m) \\ \mbox{\small s.t.} \,\,\sigma_1=\emptyset} }{H(\sigma, \emptyset; f_1, \dots, f_m)}.
    \end{equation}
Observe that Surgailis' result allows one to deduce the same conclusion without assuming that the kernels are of class $L^1$; see also \cite[Theorem 5.6]{ST24}, as well as the forthcoming discussion.

\item In relation with the content of Theorem \ref{pro:SurgaiProp3.1}, the following questions are raised in \cite[p. 222 and p. 228]{surgailis84} (see also \cite[Conjecture 4.9]{DTGS23}):
\begin{itemize}
\item[\bf (s1)] {\it Can one replace assumption \eqref{eq:SurgaAssump} with the weaker requirement that, for all $\sigma \in {\Pi}(k_1,\dots,k_m)$ and for all $A \subseteq \sigma_{\geq 2}$,
\begin{equation} \label{e:ersatz}
      {H(\sigma, A; f_1, \dots, f_m) \in L^2(\mu^{|A| + |\sigma_1|}) }\, ?
    \end{equation} 
}
\item[\bf (s2)] {\it Can one identify necessary and sufficient conditions for the property $\Phi \in L^2(\mathbb{P})$?}

\item[\bf (s3)] {\it Assume that $\Phi \in L^2(\mathbb{P})$. Is it true that relation \eqref{eq:SquareIntegrabProductMSI} holds, and that $h_K = {\rm sym}(f_1\otimes\cdots\otimes f_m)$?}

    \end{itemize}
\end{enumerate}
    }
\end{remark}

\medskip 

As discussed below, in the case $m=2$ the three questions {\bf (s1)---(s3)} have been fully addressed (and answered) in \cite[Theorem~2.2]{DoblerPeccati18} (see also \cite[Lemma~2.2]{DoblerPeccatiAOP} and \cite[Chapter~6]{privaultbook}). The aim of the present paper is to extend the results of \cite{DoblerPeccati18} to a generic $m\geq 2$. Our main findings are collected in the next statement, providing answers to {\bf (s1)---(s3)}. To the best of our knowledge, the forthcoming Theorem~\ref{t:mainintro} is the first result that closes a substantial part of the set of open questions raised in \cite{surgailis84}, in the case $m\geq 3$.

 \medskip 

\begin{theorem}\label{t:mainintro} Let the above notation prevail, and assume the setting of {\bf Problem A}. 

\smallskip

\noindent\underline{\bf Part I.} The following two properties are equivalent:

\begin{itemize}

\item[\bf (i)] $\Phi \in L^2(\mathbb{P})$;

\item[\bf (ii)] For all $q=1,\dotsc,K-1$ ($K= k_1+\cdots +k_m$) and $\mu^q$-a.e. $z_1,\dotsc,z_{q}\in\mathcal{Z}$, $D^{(q)}_{z_1,\dotsc,z_{q}}\Phi\in L^1(\P)$ and 
$(z_1,\dotsc,z_q)\mapsto\E\bigl[D^{(q)}_{z_1,\dotsc,z_{q}}\Phi\bigr]\in L^2(\mu^q)$.

\end{itemize}
Moreover, if either Condition {\bf (i)} or {\bf (ii)} is verified, then necessarily \eqref{eq:SurgaiProducDecompo} is true and the chaos decomposition \eqref{e:introchaosphi} of $\Phi$ is such that $h_K = {\rm sym}(f_1\otimes\cdots\otimes f_m)$. 

\smallskip 

\noindent\underline{\bf Part II.} Assume moreover that, for all $\sigma \in {\Pi}(k_1,\dots,k_m)$ and for all $A \subseteq \sigma_{\geq 2}$ verifying $|A|+|\sigma_1|:= q \geq 1$, one has that, 
\begin{equation}\label{e:ziocan}
\begin{array}{c}
\text{ for $\mu^q$-almost every $(z_1,...,z_q)\in \mathcal{Z}^q$, the individual kernels integrated} \\
\text{in the definition of } H(\sigma, A; f_1, \dots, f_m)(z_1,...,z_q) \text{ are of class } L^1
\end{array}
\end{equation}
(see Remark \ref{r:indy}), and that
    \begin{equation} \label{e:weakassintro}
         H(\sigma, A; f_1, \dots, f_m)(z_1,...,z_{|A|+|\sigma_1|})\,\, \mbox{is well-defined and finite a.e.-$d\mu^{|A|+|\sigma_1|}$}.
    \end{equation} 
    Then, Conditions {\bf (i)---(ii)} of {\bf Part I} are equivalent to the following requirement:
    \begin{itemize}
\item[\bf (iii)] For every $q=1,...,K-1$, the kernel $h_q$ defined in \eqref{e:hq} is an element of $L^2(\mu^q)$.
\end{itemize}
In this case, for $q=1,...,K-1$, one has that $h_q$ defined in \eqref{e:hq} is the $q$th kernel in the chaos expansion \eqref{e:introchaosphi} of $\Phi$.     
    
\end{theorem}

\medskip

\begin{remark}{\rm 

\begin{enumerate}
\item The content of Theorem~\ref{t:mainintro} for $m=2$ corresponds to \cite[Theorem~2.2]{DoblerPeccati18}. 

\item Note that Condition {\bf (ii)} of Theorem~\ref{t:mainintro} (and, when applicable, Condition {\bf (iii)}) trivially yields that $\Phi\in L^1(\P)$ but {\it does not necessarily yield} the validity of relation~\eqref{e:dumbo}, for which one needs to be able to apply Theorem~\ref{t:diagram} below (lifted from \cite[Theorem 3.1]{LPST}; see also \cite[Theorem 12.7]{LPbook}), whose assumptions are not implied by those of Theorem~\ref{t:mainintro}. As far as we know, no necessary conditions on the kernels $f_1, \dots, f_m$ ensuring $\Phi \in L^1(\mathbb{P})$ have been established to date.
 See also Remark~\ref{r:problems}-(2).

\item The results proved in \cite{DP18a} yield that, in the case $m=2$, the integrability assumption \eqref{e:ziocan} can be removed, and that assumption \eqref{e:weakassintro} is always verified for generic symmetric kernels $f_i \in L^2(\mu^{k_i})$, $i=1,2$. This implies that, in this case, the equivalence between Conditions {\bf (i)}, {\bf (ii)} and {\bf (iii)} holds in full generality. When $m\geq 3$, assumptions and \eqref{e:ziocan} and \eqref{e:weakassintro} are needed in order to represent the expectations $\E\bigl[D^{(q)}_{z_1,\dotsc,z_{q}}F\bigr]$, $q=1,...,K-1$, by means of diagram formulae---leveraging in particular the already mentioned Theorem \ref{t:diagram}. 

\item Requirements \eqref{e:ziocan} and \eqref{e:weakassintro} are equivalent to the {\bf Condition A-(loc)} introduced in Definition \ref{d:aloc} below. It is easily checked that \eqref{e:ziocan} is implied by the stronger property that
$$
f_i \in L^1(\mu^{k_i}), \quad i=1,...,m.
$$
At this level of generality, we do not expect that {\bf Condition A-(loc)} can be easily dispensed with. The reader is referred to \cite[Theorem 3.6]{ST24} for a version of Theorem~\ref{t:diagram} valid under a different set of integrability conditions, potentially yielding further variations of {\bf Part II} of Theorem \ref{t:mainintro}.


\item The analysis performed in \cite{DoblerPeccati18} also implies that, when $m=2$, the kernels $h_q$ appearing in \eqref{e:sym}, $q=1,...,K$, can be expressed in terms of {\bf contraction operators}, according to the following procedure. Assuming without loss of generality that $k_1\leq k_2$, then $h_q = 0$ if $q< k_2-k_1$ and, for $q = k_1+k_2-m$, with $m=0,...,2k_1$, one has that
\[ h_q = h_{k_1+k_2-m}=\sum_{r=\lceil \frac{m}{2} \rceil}^{m \wedge k_1} r! \binom{k_1}{r} \binom{k_2}{r} \binom{r}{m-r}\mathrm{sym} (f_1 \star_r^{m-r}f_2), \]
\end{enumerate}
where
\begin{align*}
 &(f_1\star^l_r f_2)(y_1,\dotsc,y_{r-l},t_1,\dotsc,t_{k_1-r}, s_1,\dotsc,s_{k_2-r})\\
 &:=\int_{\mathcal{Z}^l}\Bigl( f_1\bigl(x_1,\dotsc,x_l, y_1,\dotsc,y_{r-l},t_1,\dotsc,t_{k_1-r}\bigr)\\
 &\hspace{3cm}\cdot f_2\bigl(x_1,\dotsc,x_l,y_1,\dotsc,y_{r-l},s_1,\dotsc,s_{k_2-r}\bigr)\Bigr)d\mu^l(x_1,\dotsc,x_l)\,;
\end{align*}
see also \cite[Chapter 6]{privaultbook}.

}
\end{remark}

\medskip
\begin{example}[Products of single integrals]\label{exam:ProductmFunctions}{\rm
\begin{enumerate}
\item (Generic $m$) Consider setting of {\bf Problem A} with $k_1=\cdots =k_m = 1$, $m\geq 2$. To simplify the discussion, we will write $$\Pi(\underbrace{1,...,1}_{m \, {\rm times}}) := \Pi(m),$$ that is, $\Pi(m)$ is the set all partitions of $[m]$. Given $b\subseteq [m]$, we define the mapping in one variable
$$
z\mapsto f_{(b)}(z)= \left( \prod_{\ell\in b} f_\ell\right)(z):=\prod_{\ell\in b} f_\ell(z), \quad z\in \mathcal{Z}. 
$$
In this case, it is easy to check that \eqref{e:ziocan}---\eqref{e:weakassintro} are verified if and only if \begin{equation}\label{e:exo}
f_{(b)}\in L^1(\mu),\,\,  \forall b\subset [m] \,\, \mbox{such that } b\neq [m].
\end{equation}
If \eqref{e:exo} is verified, then Condition {\bf (iii)} of Theorem \ref{t:mainintro}, then, for $q=1,...,m-1$ one has that 

$$
h_q = {\rm sym}\left\{ \sum_{\substack{\sigma \in \Pi(m)\\ \sigma=\sigma_{\geq 2} \, \cup \, \sigma_1}} \; \sum_{\substack{A\subseteq \sigma_{\geq 2}\\ |A|+ |\sigma_1| = q }} \left( f_{(b_1)}\otimes\cdots \otimes f_{(b_q)} \prod_{b\in \sigma_{\geq 2} \backslash A}  \, \mu(f_{(b)}) \right)\right\}, 
$$
where after the second sum we wrote $\{b_1,...,b_q\} := A\cup \sigma_1$, and we have adopted the following notation:
$$
\mu(g) := \int_\mathcal{Z} g\, d\mu, \quad \mbox{and}\quad  \prod_\emptyset := 1.
$$
We observe that, if \eqref{e:exo} is verified also for $b=[m]$, then the forthcoming Theorem~\ref{t:diagram} yields that 
$$
\mathbb{E}[\Phi] = \sum_{\substack{\sigma = (b_1,...,b_k) \in {\Pi}(m) \\ \mbox{\small s.t.} \,\,\sigma_1=\emptyset} }\, \prod_{\ell=1}^k \mu(f_{(b_\ell)})
$$
As shown in the next two items, the special cases $m=2,3$ can be directly dealt with using {\bf Part I} of Theorem \ref{t:mainintro}.
\item ($m=2$) Specializing the content of Item 1 to the case $m=2$ and applying {\bf Part I} of Theorem \ref{t:mainintro} (or, equivalently, the results of \cite{DoblerPeccati18}), one has that $\Phi = I_1(f_1)I_1(f_2)$ is in $L^2(\P)$ if and only if $f_{([2])}=(f_1f_2) \in L^2(\mu)$ and, in this case, $h_1 = (f_1f_2)$ and $h_2= {\rm sym} (f_1\otimes f_2).$ 

\item ($m=3$) Specializing Item 1 to $m=3$ and applying {\bf Part I} of Theorem \ref{t:mainintro}, one deduces that $\Phi = I_1(f_1)I_1(f_2)I_1(f_3)$ is in $L^2(\P)$ if and only if $f_{([3])}=(f_1f_2f_3)\in L^2(\mu)$ and
$$
{\rm sym}\left\{(f_1f_2)\otimes f_3+(f_1f_3)\otimes f_2+(f_2f_3)\otimes f_1\right\} \in L^2(\mu^2).
$$
In this case, one has that
$$
h_1 = (f_1f_2f_3) + f_1\, \langle f_2,f_3\rangle_{L^2(\mu) }+f_2\, \langle f_1,f_3\rangle_{L^2(\mu) }+f_3\,  \langle f_1,f_2\rangle_{L^2(\mu) },
$$
$h_2 = {\rm sym}\left\{(f_1f_2)\otimes f_3+(f_1f_3)\otimes f_2+(f_2f_3)\otimes f_1\right\}$, and $h_3= {\rm sym} (f_1\otimes f_2\otimes f_3)$.
\end{enumerate}
}
\end{example}
\medskip

\begin{remark}{\rm It is easy to see that Condition {\bf (iii)} in {\bf Part II} of Theorem \ref{t:mainintro} is strictly weaker than requiring that \eqref{e:ersatz} is verified for all $\sigma \in {\Pi}(k_1,\dots,k_m)$ and for all $A \subseteq \sigma_{\geq 2}$.

}
\end{remark}

\medskip

Our analysis reveals that one of the main obstacles in proving Theorem~\ref{t:mainintro} is that, for $m > 2$, the random variable $\Phi$ may satisfy $\mathbb{E}|\Phi| = \infty$. To overcome this difficulty, in Section~\ref{sec:PoincInequaIntegrAssump} we establish a novel class of $p$-Poincaré inequalities, which apply to Poisson functionals that are not necessarily integrable. These estimates enable us to derive a general set of conditions under which a product of generic random variables, each living in a finite sum of Wiener chaoses, belongs to $L^p(\mathbb{P})$ for some $p \in [1,2]$. This result — presented as Theorem~\ref{t:geo} below — includes {\bf Part I} of Theorem~\ref{t:mainintro} as a particular case.

\subsection{Plan of the paper}

Section~\ref{sec:Prelimi} introduces several preliminary notions related to stochastic analysis on configuration spaces. In Section~\ref{sec:PoincInequaIntegrAssump}, we establish the announced new class of $p$-Poincar\'e inequalities. Section~\ref{ss:combone} is devoted to the analysis of iterated add-one cost operators, approached through discrete combinatorial structures associated with lattices of partitions. Finally, Section~\ref{sec:MainResu} presents a general integrability criterion for random variables expressed as products of random elements with a finite Wiener-It\^o expansion --- this result eventually leads us to complete the proof of Theorem~\ref{t:mainintro}.






\section{Preliminary notions}\label{sec:Prelimi}\label{ss:addone}

The reader is referred to \cite{privaultbook, LPbook, PTbook, Lastsv} for a complete discussion of the material presented below.

\smallskip

Consider a measurable space $(\mathcal{Z},\mathscr{Z})$, and write $\mu$ to indicate a $\sigma$-finite measure on it. According to a convention adopted e.g. in \cite{DoblerPeccatiAOP, DoblerPeccati18}, we write
\[\mathscr{Z}_\mu:=\{B\in\mathscr{Z}\,:\,\mu(B)<\infty\},\]
and use the notation 
\begin{equation*}
\eta=\{\eta(B)\,:\,B\in{ \mathscr{Z}}\}
\end{equation*}
to indicate a \textbf{Poisson random measure} on $(\mathcal{Z},\mathscr{Z})$ with \textbf{intensity} (or {\bf control}) $\mu$.  It is a well-known fact that the the distribution of $\eta$ is characterized by the following two properties: (i) for any choice $A_1,\dotsc,A_m\in\mathscr{Z}$ of pairwise disjoint measurable sets, the random variables $\eta(A_1),\dotsc,\eta(A_m)$ are stochastically independent, and (ii) for every $A\in\mathscr{Z}$, the random variable $\eta(A)$ is distributed according to a Poisson law with mean $\mu(A)$,
where we have extended the family of Poisson distributions to the completed half-line $[0,+\infty]$ in the usual way. Given $A\in\mathscr{Z}_\mu$, we use the notation
$\hat{\eta}(A):=\eta(A)-\mu(A)$ and write 
\[\hat{\eta}=\{\hat{\eta}(B)\,:\,B\in\mathscr{Z}_\mu\}\]
to indicate the \textbf{compensated Poisson measure} associated with $\eta$. Without loss of generality, one can assume that $\mathscr{F}=\sigma(\eta)$. 
Denote by $\mathbf{N}_\sigma=\mathbf{N}_\sigma(\mathcal{Z})$ the space of all $\sigma$-finite point measures $\chi$ on $(\mathcal{Z},\mathscr{Z})$ that satisfy $\chi(B)\in\N_0\cup\{+\infty\}$ for all $B\in\mathscr{Z}$. The set $\mathbf{N}_\sigma=\mathbf{N}_\sigma(\mathcal{Z})$ is equipped with the smallest 
$\sigma$-field $\mathscr{N}_\sigma:=\mathscr{N}_\sigma(\calZ)$ enjoying the property that, for every $B\in\scrZ$, the mapping $\mathbf{N}_\sigma\ni\chi\mapsto\chi(B)\in[0,+\infty]$ is measurable. For our approach, it is convenient to regard the Poisson process $\eta$ as a random element 
taking values in the measurable space $(\mathbf{N}_\sigma,\mathscr{N}_\sigma)$.

\smallskip

We also write $\mathbf{F}(\mathbf{N}_\sigma)$ to indicate the collection of all measurable functions $\mathfrak{f}:\mathbf{N}_\sigma\rightarrow\R$ and by $\mathcal{L}^0(\Omega):=\mathcal{L}^0(\Omega,\F)$ the collection of all real-valued, measurable functions $F$ on $\Omega$. Observe that, as $\F=\sigma(\eta)$, each $F\in \mathcal{L}^0(\Omega)$ can be written as $F=\mfk(\eta)$ for some measurable function $\mfk$. Such a mapping $\mfk$, often called a {\bf representative} of $F$, is $\Prob_\eta$-a.s. uniquely defined, where $\Prob_\eta=\Prob\circ\eta^{-1}$ stands for the image measure of $\Prob$ under $\eta$ on the space $(\mathbf{N}_\sigma,\mathscr{N}_\sigma)$. For $F=\mfk(\eta)\in\mathcal{L}^0(\Omega)$ and $z\in\mathcal{Z}$ we define (as in the Introduction) the \textbf{add-one cost operators} $D_z^+$, $z\in\mathcal{Z}$, as: 
\begin{equation}\label{e:addone}
D_z^{+} F:=\mfk(\eta+\delta_z)-\mfk(\eta)\,.
\end{equation}
One immediately verifies the following product rule: for $F,G\in \mathcal{L}^0(\Omega)$ and $z\in\mathcal{Z}$ one has
\begin{align}\label{prodrule}
D^+_z(FG)&=GD^+_zF+FD_z^+G+D_z^+F D_z^+G\,.
\end{align}
More generally, if $m\in\N$ and $z_1,\dotsc, z_m\in\mathcal{Z}$, then we define inductively  $D_{z_1}^{(1)}=D_{z_1}^+$ and
\begin{equation*}
D_{z_1,\dotsc,z_m}^{(m)} F:=D_{z_1}^+\bigl(D_{z_2,\dotsc,z_m}^{(m-1)}F\bigr)\,,\quad m\geq2\,.
\end{equation*}
Writing $[m]:=\{1,\dots,m\}$, it is easily seen that 
\begin{equation}\label{Dm}
D_{z_1,\dotsc,z_m}^{(m)} F =\sum_{J\subseteq[m]}(-1)^{m-\abs{J}}\; \mfk\Bigl(\eta+\sum_{i\in J}\delta_{z_i}\Bigl)
\end{equation}
which shows that the mapping $\Omega\times\mathcal{Z}^m\ni (\omega,z_1,\dotsc,z_m)\mapsto D_{z_1,\dotsc,z_m}^{(m)} F(\omega)
\in\R$ is $\F\otimes\mathscr{Z}^{\otimes m}$-measurable. Moreover, it also implies that $D_{z_1,\dotsc,z_m}^{(m)} F=D_{z_{\sigma(1)},\dotsc,z_{\sigma(m)}}^{(m)} F$ for each permutation $\sigma$ of $[m]$. We observe that, e.g. by virtue of \cite[Lemma 2.4]{LaPen11}, the definition of $D^{(q)}F$ is $\Prob\otimes \mu^q$-a.e. independent of the choice of the representative $\mfk$.
\smallskip

We use the symbol $L$ to denote the {\bf generator of the Ornstein-Uhlenbeck semigroup} associated with $\eta$, whereas $\dom L\subseteq L^2(\P)$ indicates its domain (see \cite[p.~21]{Lastsv}). It is well-known that $-L$ is a symmetric, diagonalizable operator on $L^2(\P)$, having the pure point spectrum 
$\N_0=\{0,1,\dotsc\}$. For $q\in\N_0$, we denote by $C_q:=\ker(L+q{\rm Id})$ the so-called \textbf{$q$-th Wiener chaos} {associated with }$\eta$, where we write ${\rm Id}$ to indicate the identity operator on $L^2(\P)$. One can show that, for $q\in\N$, the linear space $C_q$ is the collection of all \textbf{multiple Wiener-It\^{o} integrals} $I_q(h)$ of order $q$ with respect to ${ \hat{\eta}}$, as defined e.g. in \cite[Section 3]{Lastsv}, where $h$ is a square-integrable function on the product space 
$(\mathcal{Z}^q,\mathscr{Z}^{\otimes q}, \mu^q)$. For 
$c\in\R$ we also let $I_0(c):=c$ in such a way that $C_0=\{I_0(c)\,:\, c\in\R\}$. Multiple integrals enjoy two fundamental properties.
Let $q,p\geq0$ be integers: then,
\begin{enumerate}
 \item $I_q(h)=I_q({\rm sym}(h))$, where we have used the notation \eqref{e:sym};
 \item $I_q(h)\in L^2(\Prob)$, and $\E\bigl[I_q(h)I_p(g)\bigr]= \delta_{p,q}\,q!\,\left\langle {\rm sym}(h),{\rm sym}(g)\right\rangle $, where $\delta_{p,q}$ denotes {Kronecker's delta symbol}, and $\langle \cdot, \cdot \rangle$ is the usual inner product in $L^2(\mu^q)$.
\end{enumerate}
 For an integer $q\geq1$ we write $L^2(\mu^q)$ to indicate the Hilbert space of all (equivalence classes of) square-integrable and real-valued measurable functions on $\mathcal{Z}^q$ and we write $L^2_s(\mu^q)$ for the subspace of those functions in $L^2(\mu^q)$ which are $\mu^q$-a.e. symmetric. Moreover, to simplify notation, we denote by $\Enorm{\cdot}$ and $\langle \cdot,\cdot\rangle$ the usual norm and scalar product 
on $L^2(\mu^q)$, irrespective of the value of $q$. We also set $L^2(\mu^0):=\R$. If $F=I_q(f)$ for some $q\geq 1$ and $f\in L_s^2(\mu^q)$, then for $\mu$-a.e. $z\in\mathcal{Z}$ one has 
\begin{equation}\label{derint}
 D_z^+F=q I_{q-1}\bigl(f(z,\cdot)\bigr),\quad \mbox{a.s.}-\Prob.
\end{equation}
In particular, $D_z^+F$ is a multiple Wiener-It\^{o} integral of order $q-1$. If $q=0$, then it is easy to see that $D^+_z F=0$.

As already recalled, it is a fundamental fact that every $\Phi\in L^2(\Prob)$ admits 
a unique representation 
\begin{equation}\label{chaosdec}
 \Phi=\E[\Phi]+\sum_{q=1}^\infty I_q(h_q)\,,
\end{equation}
where $h_q\in L_s^2(\mu^q)$, $q\geq1$, are suitable symmetric integrands, and the series converges in $L^2(\Prob)$. Identity \eqref{chaosdec} is known as the \textbf{Wiener-It\^o chaos decomposition} of $\Phi \in L^2(\Prob)$. Relation \eqref{chaosdec} entails the abstract decomposition 
\begin{equation*}
L^2(\P)=\bigoplus_{q=0}^\infty C_q\,,
\end{equation*}
where the sum on the right-hand side is orthogonal in $L^2(\Prob)$. 



\section{A new class of $p$-Poincar\'e inequalities}\label{sec:PoincInequaIntegrAssump}

As anticipated in the Introduction, one of the crucial technical elements in our approach is a collection of new $p$-{\bf Poincar\'e inequalities} on the Poisson space --- only requiring the almost sure finiteness of the involved random variables. This result, stated in Theorem \ref{t:ppi} below, generalizes part of the estimates established in \cite{trauthwein} as well as the classical $L^1$ and $L^2$ Poincar\'e inequalities on the Poisson space stated e.g.~in \cite[Theorem 18.7 and Corollary 18.8]{LPbook}. We recall one of the main estimates from \cite{trauthwein}.

\smallskip 

\begin{theorem}[Formula (4.7) in \cite{trauthwein}]\label{t:tara} Let $\gamma$ be a Poisson measure on a measurable space $(\mathcal{A} , \mathscr{A})$, with $\sigma$-finite intensity $\nu$. Assume that $G = G(\gamma)$ is such that $\mathbb{E}|G| < \infty$. Then, for all all $p\in [1,2]$ one has that
\begin{equation}\label{e:tara}
\mathbb{E}\left| G \right|^p  \leq \left| 
\mathbb{E}[G] \right| ^p +  2^{2-p} \int_\mathcal{A} \mathbb{E}[ |D_a^+ G|^p]\, \nu(da).
\end{equation}
    
\end{theorem}

Note that \eqref{e:tara} implies that, if the integral on the right-hand side of \eqref{e:tara} is finite, then $\mathbb{E} |G|^p <\infty$. The next result allows one to extend the content of Theorem~\ref{t:tara} to the case of a.s.~finite (and not necessarily integrable) random variables. From now on, we let the notation and assumptions of Section \ref{ss:addone} prevail.

\smallskip 

\begin{theorem}[\bf $p$-Poincar\'e inequalities for a.s.~finite variables]\label{t:ppi}  Suppose that $F\in \mathcal{L}^0(\Omega)$, so that $\mathbb{P}(|F|<\infty) = 1$. Then, for all $p\in [1,2]$,
\begin{equation}\label{e:ppi}
\mathbb{E}\left| F - F' \right|^p  \leq 2^{3-p} \int_\mathcal{Z} \mathbb{E}[ |D_z^+ F|^p]\, \mu(dz),
\end{equation}
where $F'$ is an independent copy of $F$. In particular, if the right-hand side of \eqref{e:ppi} is finite for some $p\in [1,2]$, one has that $F\in L^{p}(\mathbb{P})$.
\end{theorem}
\begin{proof} For the rest of the proof we fix a representative $\mfk$ of $F$. Consider the product space $ \mathcal{A} =\{0,1\}\times \mathcal{Z}$, let $\nu$ denote the measure on $\mathcal{A}$ given by $\nu(dj, dz) := (\delta_0(dj) + \delta_1(dj))\mu(dz)$, and write $\gamma$ to indicate a Poisson measure on $\mathcal{A}$ with intensity $\nu$. Then, the mappings $B\mapsto \gamma_i(B) :=  \gamma(\{i\}\times B)$, $i=0,1$, define two independent copies of $\eta$. Without loss of generality, we can now assume that $F = \mfk(\gamma_0)$ and $F' = \mfk(\gamma_1)$. Writing $G = F-F'$ one has that, for an arbitrary $(i,z)\in \mathcal{A}$, $D^+_{(i,z)}G = \mfk(\gamma_0+ \delta_z ) - \mfk(\gamma_0) $ if $i=0$ and $D^+_{(i,z)}G = \mfk(\gamma_1)-\mfk(\gamma_1+ \delta_z )$, if $i=1$. For $s>0$ and $i=0,1$ we set $\mfk_s(\gamma_i) := -s{\bf 1}_{\mfk (\gamma_i)<-s} + \mfk(\gamma_i){\bf 1}_{-s\leq \mfk(\gamma_i)\leq s} + s {\bf 1}_{\mfk(\gamma_i)>s}$. One has that: (a) for all $s>0$, $G_{(s)}:= \mfk_s(\gamma_0) - \mfk_s(\gamma_1)$ is a bounded random variable with zero expectation, (b) applying \eqref{e:tara} to $G = G_{(s)}$ yields
\begin{eqnarray*}
 \mathbb{E}[ | G_{(s)} |^p]  &\leq &  2^{2-p} \sum_{i=0,1} \int_{\mathcal{Z}} \mathbb{E} [| \mfk_s(\gamma_i + \delta_z)- \mfk_s(\gamma_i) |^p] \, \mu(dz) \\
&=& 2^{3-p}\!\!\int_{\mathcal{Z}} \mathbb{E} [| \mfk_s(\gamma_1 + \delta_z)- \mfk_s(\gamma_1) |^p] \, \mu(dz),
\end{eqnarray*}
and (c) $G_{(s)} \to G$, a.s.-$\mathbb{P}$, as $s\to\infty$. Since $|\mfk_s(\gamma_1 + \delta_z)- \mfk_s(\gamma_1) |^p\leq |\mfk(\gamma_1 + \delta_z)- \mfk(\gamma_1) |^p $, a.s.-$\mathbb{P}$, and the latter quantity has the same distribution as $|D_z^+ F|^p$, we infer from Fatou's Lemma the desired relation \eqref{e:ppi}. 
To conclude, we observe that, if the right-hand side of \eqref{e:ppi} is finite, then independence and Fubini's theorem imply that, for at least one $x\in \R$, one has that
$$
\mathbb{E} [ |F - x|^p]<\infty,
$$
and the triangle inequality immediately yields that $F\in L^{p}(\mathbb{P})$.
\end{proof}

We recall a fundamental result originally proved in \cite[Theorem 1.3]{LPfock}.

\smallskip 

\begin{theorem}[{\bf Last-Penrose formula} \cite{LPfock}]\label{t:lastpenrose} Let $F\in L^2(\Prob)$. For all $q\geq1$ and for $\mu^q$-a.e. $z_1,\dotsc,z_q\in\calZ$, one has that $D^{(q)}_{z_1,\dotsc,z_q}F\in L^1(\Prob)$. Moreover, the kernel $h_q$ in \eqref{chaosdec} can be taken to be 
\begin{equation}\label{kerform}
 h_q(z_1,\dotsc,z_q)=\frac{1}{q!}\E\bigl[D^{(q)}_{z_1,\dotsc,z_q}F\bigr],
\end{equation}
for all $z_1,\dotsc,z_q\in\calZ$ such that the right-hand side of \eqref{kerform} is finite.     
\end{theorem}

\smallskip

The following statement is a substantial generalization of \cite[Lemma 5.1]{DoblerPeccati18}, and is one of the main tools used in our work.

\smallskip

\begin{proposition}\label{p:doblerpeccati}
Fix $p\in [1,2]$. Suppose that $F$ is a $\sigma(\eta)$-measurable random variable such that $\mathbb{P}(|F|<\infty) = 1$ and that there exists $M\geq 1$ such that:
\begin{enumerate}
\item[\rm (A)] For all $z_1,\dotsc,z_{M+1}\in\mathcal{Z}$ one has $D^{(M+1)}_{z_1,\dotsc,z_{M+1}}F=0$, a.s.-$\mathbb{P}$.
\item[\rm (B)] For all $q=1,\dotsc,M$ and $\mu^q$-a.e. $z_1,\dotsc,z_{q}\in\mathcal{Z}$, $D^{(q)}_{z_1,\dotsc,z_{q}}F\in L^1(\P)$\\ and 
$(z_1,\dotsc,z_q)\mapsto\E\bigl[D^{(q)}_{z_1,\dotsc,z_{q}}F\bigr]\in L^p(\mu^q)$.
\end{enumerate}
Then, $F\in L^p(\P)$.  
\end{proposition}
\begin{proof} Iterating $M$ times \eqref{e:tara}, one sees that Assumptions (A) and (B) in the statement imply that the quantity $$
\int_\mathcal{Z} \mathbb{E}[ |D_z^+ F|^p]\, \mu(dz)
$$
is bounded by 
$$
2^{M(2-p)} \sum_{q=1}^M \int_{\mathcal{Z}^q} \left| \mathbb{E}[ D_{z_1,\dots,z_q}^{(q)} F] \right|^p \mu(dz_1)\cdots \mu(dz_q) <\infty, 
$$
and the conclusion follows from Theorem \ref{t:ppi}. 
\end{proof}

In what follows, given integers $0\leq i\leq k$, we will use the {\bf falling factorial} symbol $k_{(i)}$, which is defined as $k_{(0)} := 1$, and $k_{(i)} := k(k-1)\cdots (k-i+1)$, when $i\geq 1$.

\section{Combinatorial representation of add-one costs for products}\label{ss:combone}

\subsection{General formulae} Let $m \geq 1$ be an integer and recall that $[m]= \{1,\dots,m\}$. For every subset $\emptyset\neq A\subseteq [m]$ and every $z\in \mathcal{Z}$, we define the mapping 
$$
D_z^A : \underbrace{\mathcal{L}^0(\Omega)\times \cdots \times \mathcal{L}^0(\Omega)}_{\mbox{\textit{m} times}} := \mathcal{L}^0(\Omega)^m  \longrightarrow \mathcal{L}^0(\Omega)^m 
$$
as follows: for all $(F_1,\dots,F_m)\in \mathcal{L}^0(\Omega)^m$, $D_z^A(F_1,\dots,F_m) = (X_1,\dots,X_m)$, where $X_i = D^+_z F_i$ if $i\in A$ and $X_i = F_i$ otherwise. For instance, if $m= 3$ and $A= \{1,3\}$, then $D_z^A(F_1,F_2,F_3) = (D^+_zF_1, F_2, D^+_z F_3)$. For the rest of the section, we will denote by $Q: \mathcal{L}^0(\Omega)^m \to \mathcal{L}^0(\Omega)$ the usual pointwise multiplication operator given by $Q(X_1,\dots,X_m) := \prod_{i=1}^m X_i$. 
\smallskip

The following lemma is a straightforward extension of \eqref{prodrule} and can be proved by recursion (details are left to the reader).

\begin{lemma}\label{l:dmprod} For a generic $(F_1,\dots,F_m)\in \mathcal{L}^0(\Omega)^m$, set $\Phi := Q(F_1,..,F_m) = \prod_{i=1}^m F_i$. Then, for all $z\in \mathcal{Z}$ one has that
\begin{equation}\label{e:demprod}
D^+_z \Phi = \sum_{\emptyset\neq A \subseteq [m]} Q(D_z^A (F_1,\dots,F_m)).
\end{equation}

\end{lemma}
We will now introduce a formalism that generalizes the approach initiated in \cite{DoblerPeccati18}. Fix $m\geq 1$ as above. Given $q\geq 1$, a {\bf word} $W$ of length $|W| = q$ in the {\bf alphabet} $\{A \subseteq [m] \mid  A \neq \emptyset\}$ is a vector $W = (A_1,\dots,A_q)$ of non-empty subsets of $[m]$. Given a word $W = (A_1,\dots,A_q)$ and $z_1,\dots,z_q \in \mathcal{Z}$, we set $D_{z_1}^{[W]}(F_1,\dots,F_m) := D_{z_1}^{A_1}(F_1,\dots,F_m)$ if $q=1$, and
\begin{equation}\label{e:wordp}
D_{z_1,\dots,z_q}^{[W]}(F_1,\dots,F_m) :=D_{z_1}^{A_1} ( D^{[W']}_{z_2,\dots, z_q}(F_1,\dots,F_m)), 
\end{equation}
where $W' := (A_2,\dots,A_q)$. Using \eqref{prodrule}, \eqref{e:demprod} and a recursion argument we infer that, given $(F_1,\dots,F_m)\in \mathcal{L}^0(\Omega)^m$ and setting $\Phi := Q(F_1,..,F_m) = \prod_{i=1}^m F_i$, one has that, for all $q\geq 1$,
\begin{equation} \label{e:dpprod}
D^{(q)}_{z_1,\dots,z_q} \Phi = \sum_{|W| = q } Q(D_{z_1,\dots,z_q}^{[W]}(F_1,\dots,F_m)),
\end{equation}
where $Q$ is the multiplication operator defined above and the sum runs over all words with length $q$. The following statement is a direct consequence of \eqref{derint} and \eqref{e:dpprod}.

\smallskip 

\begin{lemma}\label{l:order} Fix $m\geq 2$ and consider (not necessarily distinct) integers $k_1,\dots,k_m\geq 0$. Let $F_1,\dots,F_m \in L^2(\Prob)$ be such that
$$
F_i \in \bigoplus_{q = 0}^{k_i} C_q, \quad i=1,\dots,m,
$$
and set $\Phi := \prod_{i=1}^m F_i$. Then, for all $M>k_1+\cdots+ k_m$, one has that $D^{(M)}_{z_1,\dots,z_M} \Phi = 0$ for $\mu^M$-a.e. $z_1,..,z_M$.   
\end{lemma}
\begin{proof} From \eqref{derint} one deduces immediately that, for all $\ell>k_i$, one has that $D^{(\ell)}_{z_1,\dots,z_\ell}F_i = 0$ for $\mu^\ell$-a.e. $z_1,\dots,z_\ell$. The conclusion is obtained by observing that, if $W$ is a word of length $M>k_1+\cdots +k_m$, then the random variable $Q(D_{z_1,\dots,z_M}^{[W]}(F_1,\dots,F_m))$ is the product of $m$ factors of which at least one has the form $D^{(\ell)}_{z_{i_1},\dots,z_{i_\ell}}F_i$, for some $\{z_{i_1},\dots,z_{i_\ell}\}\subset \{z_1,\dots,z_M\}$ and some $k_i<\ell\leq M$.
\end{proof}

\begin{remark}\label{r:restricted} {\rm 
Reasoning as in the previous proof, one sees that, for $F_1,\dots,F_m$ as in the statement of Lemma \ref{l:order} and for every $q\geq  1$, the sum in \eqref{e:dpprod} can be taken to be over the smaller set 
 ${\bf W}(q \, ; \, k_1,\dots,k_m)$, defined as the collection of all words $(A_1,\dots,A_q)$ in the alphabet $\{A\subseteq [m] \mid A\neq\emptyset \}$ such that, for all $i=1,\dots,m$,
 \begin{equation}\label{e:restricted}
d_i= d_i(A_1,..,A_q) := \Big|\{ \ell \in [q] : i\in A_\ell\}\Big| \leq k_i.  
\end{equation}
Note that, consistently with Lemma \ref{l:order}, one has that ${\bf W}(q \, ; \, k_1,\dots,k_m) = \emptyset$ if $q> k_1+\cdots+k_m$.
}
\end{remark}

\medskip

The next section contains some further combinatorial notions, that are useful to deal with the situation in which each $F_i$ in Lemma \ref{l:order} is an element of a single Wiener chaos.

\subsection{The language of partitions and contractions}

\subsubsection{Partitions, tensors and expectations} \label{ss:partitions} Fix $m\geq 1$ and a vector $(k_1,\dots,k_m)\in \mathbb{Z}_+^m:= \{1,2,\dots\}^m$, write $K := k_1+\cdots +k_m$, and adopt the notation and conventions introduced in Section \ref{ss:existintro}. We write $\Pi_{\geq 2} (k_1,\dots,k_m):= \{ \sigma \in \Pi(k_1,\dots,k_m) : |b|\geq 2, \forall b\in \sigma\}$, that is: $\Pi_{\geq 2}(k_1,\dots,k_m)$ is the subset of $\Pi(k_1,\dots,k_m)$ composed of partitions with blocks at least of size 2. Given a partition $\sigma$ we denote by $|\sigma|$ the {\bf size} (that is, the number of blocks) of $\sigma$. In what follows, we will sometimes need to extend the definitions of $\Pi(k_1,\dots,k_m)$ and $\Pi_{\geq 2}(k_1,\dots,k_m)$ to the case in which some of the integers $k_i$ are equal to zero. To this end, given a vector of nonnegative integers $(k_1,\dots,k_m)\in \mathbb{N}_0^m:= \{0,1,...\}^m$, we set $\Pi(k_1,\dots,k_m)=\Pi_{\geq 2}(k_1,\dots,k_m):= \emptyset$ if $k_1=\cdots =k_m = 0$ and otherwise we define  $\Pi(k_1,\dots,k_m):=\Pi(k'_1,\dots,k'_\ell)$ and $\Pi_{\geq 2}(k_1,\dots,k_m):=\Pi_{\geq 2}(k'_1,\dots,k'_\ell)$, where $(k'_1,\dots,k'_\ell)$ ($\ell\leq m$) is the subvector of $(k_1,\dots,k_m)$ composed of strictly positive integers. Note that the elements of $\Pi(k_1,\dots,k_m)$ and $\Pi_{\geq 2}(k_1,\dots,k_m)$ are {\bf partitions} of $[K]$ where \begin{equation} \label{e:K}
K := k'_1+\cdots+ k'_\ell = k_1+\cdots + k_m
\end{equation}
as before.

\medskip

\noindent{\bf Notation.} Given a vector of nonnegative integers $(k_1,\dots,k_m)\in \mathbb{N}_0^m$, consider symmetric kernels $f_i\in L^2_s (\mu^{k_i})$, $i=1,\dots,m$, and observe that, if $k_i = 0$, then $f_i$ is a real constant. As before, we define the {\bf tensor product} $f_1\otimes \cdots \otimes f_m$ to be the mapping from $\mathcal{Z}^K$ into $\mathbb{R}$ (where $K$ is defined in \eqref{e:K}) given by
\begin{equation}\label{e:tensor}
{f_1\otimes \cdots \otimes f_m (v_1,\dots,v_K) :=  \prod_{i=1}^m f_i(v_{k_1+\cdots+ k_{i-1}+1},\dots, v_{k_1+\cdots +k_i}), \quad (v_1,\dots,v_K) \in \mathcal{Z}^K,}
\end{equation}
where, on the right-hand side of the previous equation, the factor corresponding to $i=1$ is equal to {\color{blue} $f_1(v_1,\dots,v_{k_1})$} by convention and, for $i=1,\dots,m$ 
$$
{f_i(v_{k_1+\cdots+ k_{i-1}+1},\dots, v_{k_1+\cdots +k_i}) \equiv f_i \in \mathbb{R},}
$$
if $k_i = 0$. As in \cite[p 116]{LPbook}, given $\sigma \in \Pi(k_1,\dots,k_m)$ we use the symbol $(f_1\otimes \cdots \otimes f_m)_\sigma$ to indicate the real-valued mapping on $\mathcal{Z}^{|\sigma|}$ obtained from $f_1\otimes \cdots \otimes f_m $ by identifying two variables in the argument of $f_1\otimes \cdots \otimes f_m $ if and only if they belong to the same block of $\sigma$. For instance:
\begin{itemize}
    \item[---] if $m=2$, $k_1 = 1$, $k_2 = 2$ and $\sigma = \{\{1,2\} , \{3\}\}$, then $(f_1\otimes f_2)_\sigma (z_1,z_2) = f_1(z_1) f_2(z_1, z_2)$;
    \item[---] if $m=3$, $k_1=k_2=k_3 = 2$ and $\sigma = \{\{1,3,5\}, \{2\}, \{4,6\}\}$, then $(f_1\otimes f_2\otimes f_3)_\sigma (z_1,z_2,z_3) = f_1(z_1, z_2) f_2(z_1, z_3) f_3(z_1,z_3)$;
    \item[---] if $m=3$, $k_1=k_3=2$ and $k_2 = 0$ (so that $(k'_1, k'_2) = (k_1, k_3) = (2,2)$) and $\sigma = \{\{1,3\}, \{2,4\}\}$, then $(f_1\otimes f_2\otimes f_3)_\sigma (z_1,z_2) = f_1(z_1, z_2)\cdot f_2 \cdot  f_3(z_1,z_2)$, where $f_2$ is a real constant.
\end{itemize}
Note that $|(f_1\otimes \cdots \otimes f_m)_\sigma| =(|f_1 |\otimes \cdots \otimes | f_m| )_\sigma$, by definition. Also, when $k_1,...,k_m\geq 1$ and $\sigma\in \Pi_{\geq 2}(k_1,...,k_m)$, one can easily relate the function $(f_1\otimes \cdots \otimes f_m)_\sigma$ with the notation introduced in Section \ref{ss:existintro} through the identity
$$
\int_{Z^{|\sigma|}}(f_1\otimes \cdots \otimes f_m)_\sigma d \mu^{|\sigma|}= H(\sigma, \emptyset; f_1,...,f_m).
$$
The notation adopted in the present section is meant to facilitate the connection with references \cite{LPbook, SchulteThaele12, SchulteThale2024}.

\medskip

\begin{definition}\label{d:ca} {\rm Consider $(k_1,\dots,k_m)\in \mathbb{N}^m$, as well as symmetric kernels $f_i\in L^2_s (\mu^{k_i})$, $i=1,\dots,m$. We say that the kernels $f_1,\dots, f_m$ verify {\bf Condition A} if either (a) $k_1=\cdots = k_m = 0$, or (b) $f_i\in L^1(\mu^{k_i})$ for all $i = 1,\dots,m$ and, for every $\sigma \in \Pi(k_1,\dots,k_m)$, one has that
\begin{equation}\label{e:intsigma}
\int_{\mathcal{Z}^{|\sigma|}} \Big| (f_1\otimes \cdots \otimes f_m)_\sigma \Big| d\mu^{|\sigma |} <\infty.
\end{equation}
}  
\end{definition}

The following result provides sufficient conditions for a product of multiple Wiener-It\^o integrals to be in $L^1(\mathbb{P})$ and also yields an explicit (combinatorial) expression for its expectation. A complete proof is given in \cite[Theorem 12.7]{LPbook}.

\smallskip 

\begin{theorem}\label{t:diagram} Consider a vector $(k_1,\dots,k_m)\in \mathbb{N}^m$, as well as symmetric kernels $f_i\in L^2_s (\mu^{k_i})$, $i=1,\dots,m$. If the kernels $f_1,\dots,f_m$ satisfy {\rm \bf Condition A} (see Definition \ref{d:ca}), then $\Phi := \prod_{i=1}^m I_{k_i}(f_i)\in L^1(\mathbb{P})$, and moreover
\begin{equation}\label{e:diagrams}
\mathbb{E}[\Phi] = \sum_{\sigma\in \Pi_{\geq 2} (k_1,\dots,k_m) } \int_{\mathcal{Z}^{|\sigma|}}  (f_1\otimes \cdots \otimes f_m)_\sigma \,  d\mu^{|\sigma |},
\end{equation}
where the right-hand side of equation \eqref{e:diagrams} equals by definition the product $f_1\cdots f_m \in \mathbb{R} $ in case $k_1=\cdots = k_m = 0$.
\end{theorem}

\medskip 

\begin{remark}\label{r:diagrams}{\rm \begin{enumerate}
    \item Assume that the symmetric kernels $f_1,\dots,f_m$ satisfy {\bf Condition~A}, but are not necessarily elements of $L^2(\mu^{k_i})$ ($i=1,\dots,m$). It is well-known that, in this case, one can still define the multiple integrals $I_{k_i}(f_i)$ using the analytical definition for integrable kernels given e.g.~in \cite[eq. (12.12)]{LPbook}, and that the conclusion of Theorem \ref{t:diagram} remains valid --- see \cite[Proposition 12.6 and Thereom 12.7]{LPbook} for a discussion of this point. We also recall that, if for some $k\geq 1$ $f\in L^2(\mu^k) \cap L^1(\mu^k)$, then the definition of the multiple integral $I_k(f)$ adopted in this paper and that of \cite[eq. (12.12)]{LPbook} coincides with probability one; see \cite[Proposition 12.9]{LPbook}.
    \item {\bf Condition A} is {\it not} necessary for the conclusion of Theorem \ref{t:diagram} to hold. Consider for instance the case $m=2$, as well as two kernels $f_i \in L^2_s(\mu^{k_i})$, $i=1,2$ (not necessarily satisfying {\bf Condition A}): then, by Cauchy-Schwarz one has always that $\Phi := I_{k_1}(f_1)I_{k_2}(f_2)\in L^1(\mathbb{P})$ and the usual isometry formula $\mathbb{E}[\Phi] = {\bf 1}_{k_1=k_2} \, k_1! \langle f_1, f_2\rangle_{L^2(\mu^{k_1})}$ holds (it is a standard exercise to show that such an isometric relation is equivalent to \eqref{e:diagrams} in this case). 
    \item A special case of Theorem \ref{t:diagram} is stated in \cite[Corollary 7.4.1]{PTbook}.
\end{enumerate}}
    \end{remark}

\subsubsection{Local conditions} Fix $m\geq 2$, and consider $(k_1,\dots, k_m)\in \mathbb{N}_0^m $, as well as kernels $f_i\in L^2(\mu^{k_i})$, $i=1,\dots,m$. For every $q = 1,\dots, K := k_1+\cdots + k_m$, we define the class ${\bf W}(q \, ; \, k_1,\dots,k_m)$ according to Remark \ref{r:restricted}, and adopt the notation \eqref{e:restricted}. Given $W =(A_1,\dots,A_q)\in {\bf W}(q \, ; \, k_1,\dots,k_m)$, the mapping $(z_1,\dots,z_q) \mapsto Q(D^{[W]}_{z_1,\dots,z_q}(I_{k_1}(f_1),\dots, I_{k_m}(f_m))$, defined according to the conventions of Section \ref{ss:combone}, admits the following representation (which is a direct consequence of \eqref{derint}): for $\mu^q$-almost every $(z_1,\dots,z_q)$,
\begin{equation}\label{e:prediagram}
Q(D^{[W]}_{z_1,\dots,z_q}(I_{k_1}(f_1),\dots, I_{k_m}(f_m)) = \prod_{i=1}^m (k_i)_{(d_i)} \cdot I_{k_i - d_i} (f_i( z_{{\bf q}(i)}, \cdot)),
    \end{equation}
where we have adopted the following conventions: (a) the symbol $(k_i)_{(d_i )}$ denotes the usual falling factorial; (b) each ${\bf q}(i)$, $ i=1,\dots,m$, is defined as the (possibly empty) set $\{j_1^{(i)},\dots,j_{d_i}^{(i)}\}$ of those $j \in [q]$ such that $i\in A_j$ (in such a way that $|{\bf q}(i) | = d_i $); (c) for $i=1,\dots,m$, the kernel $f_i( z_{{\bf q}(i)}, \cdot)$ is the element of $L^2_s(\mu^{k_i-d_i})$ defined by the mapping
$$
(a_1,\dots,a_{k_i-d_i} )\mapsto f_i \left(z_{j^{(i)}_1},\dots,z_{j^{(i)}_{d_i}}, a_1,\dots,a_{k_i-d_i}\right),
$$
that is, $f_i( z_{{\bf q}(i)}, \cdot)$ is obtained from $f_i$ by fixing the first $d_i$ elements of its argument to be equal to those entries $z_j$ of the vector $(z_1,\dots, z_q)$ such that $i\in A_j$. We stress that the kernels $f_i$ are symmetric, and that the above definitions are therefore robust with respect to arbitrary permutations of the arguments of the kernels $f_i$.
\medskip

\begin{example}\label{ex:wick}
{\rm 
\begin{enumerate}
\item Consider the case $m=3$, $k_1=k_2=k_3 = 2$, $q=2$ and $W =(A_1,A_2) = ([3], \{1,3\})$. Then,
$$
Q(D^{[W]}_{z_1,z_2}(I_{2}(f_1),I_{2}(f_2), I_{2}(f_3))= 8 f_1(z_1,z_2)f_3(z_1,z_2) I_1(f_2(z_1, \cdot)).
$$
\item Writing $K = k_1+\cdots + k_m$, one has that ${\bf W}(K \, ; \, k_1,\dots,k_m)$ coincides with the collection of those words $(A_1,\dots,A_K)$ such that $d_i = k_i$ for each $i=1,\dots,m$. Defining ${\bf U}(k_1,\dots,k_m)$ to be the set of all vectors of the form $(B_1,\dots,B_m)$ such that the sets $B_i\subset [K]$ are pairwise disjoint, $|B_i| = k_i$ and $\cup_i B_i = [K]$, one sees that there exists a bijection $\varphi$ from ${\bf W}(K \, ; \, k_1,\dots,k_m)$ onto ${\bf U}(k_1,\dots,k_m)$ with the following property: if $(B_1,\dots,B_m) = \varphi(W)$, then 
$$
Q(D^{[W]}_{z_1,\dots,z_K}(I_{k_1}(f_1),\dots, I_{k_m}(f_m)) = \prod_{i=1}^m k_i!\cdot f_i({\bf z}_{B_i}),
$$
where ${\bf z}_{B_i}$ is the vector composed of those $z_k$ such that $k\in B_i$. This yields in particular that 
\begin{eqnarray}
\notag && \frac{1}{K!}\sum_{|W| = K } Q(D^{[W]}_{z_1,\dots,z_K}(I_{k_1}(f_1),\dots, I_{k_m}(f_m)) \\ \notag &&= \frac{1}{K!}\sum_{W\in {\bf W}(K \, ; \, k_1,\dots,k_m) } Q(D^{[W]}_{z_1,\dots,z_K}(I_{k_1}(f_1),\dots, I_{k_m}(f_m))\\
&& = {\rm sym} (f_1\otimes \cdots \otimes f_m)(z_1,\dots,z_K). \label{e:symtens}
\end{eqnarray}

\end{enumerate}
}
\end{example}

\medskip

\begin{definition}\label{d:aloc}{\rm Fix $m\geq 2$, consider integers $k_1,\dots, k_m\geq 1$ and kernels $f_i\in L^2(\mu^{k_i})$, $i=1,\dots,m$, and let the above notation and terminology prevail (in particular, we write $K := k_1+\cdots +k_m$). We say that the kernels $f_1,\dots,f_m$ verify {\bf Condition A-(loc)} if, for every $q=1,\dots, K-1$, for every word $W =(A_1,\dots,A_q)\in {\bf W}(q \, ; \, k_1,\dots,k_m)$, and for $\mu^q$-almost every $(z_1,\dots,z_q)$, one has that the kernels $f_i( z_{{\bf q}(i)}, \cdot)$ verify {\bf Condition A}, as formalized in Definition \ref{d:ca}.
}
\end{definition}

\smallskip
\begin{remark}\label{r:alocc}{\rm  As anticipated in the Introduction, it is a standard exercise to verify that {\bf Condition {A}-(loc)} is equivalent to \eqref{e:ziocan} and \eqref{e:weakassintro}.
    }
\end{remark}

\smallskip
\begin{remark}{\rm 
    It is tedious but straightforward to show that {\bf Condition~A} implies {\bf Condition~A-(loc)} by Fubini's theorem, and that the reciprocal implication is false in general. To see this last point, consider the case $\mathcal{Z} = (1, +\infty) $, $\mu(dz) = z^{-5/2}$, $m=3$, $k_1=k_2=k_3 = 1$ and {$f_1(v)= f_2(v) = f_3(v) = v^{1/2}$}. We have that $\sigma=\{\{1,2,3\}\} \in  \Pi(1,1,1) $ and $\int_{\mathcal{Z}} |(f_1 \otimes f_2\otimes f_3)_{\sigma}| d\mu =\infty$, whereas {\bf Condition~A-(loc)} holds. }
\end{remark}

\medskip

The following useful statement is obtained by combining  \eqref{e:dpprod}, \eqref{e:diagrams} and \eqref{e:prediagram}.

\medskip

\begin{proposition}\label{prop: OneCostCondiA} Consider $(k_1,\dots,k_m)\in \mathbb{Z}_+^m$, as well as symmetric kernels $f_i\in L^2_s (\mu^{k_i})$, $i=1,\dots,m$. Assume that the kernels $f_i, \,i=1,\dots,m,$ verify {\bf Condition A-(loc)} and write $\Phi = \prod_{i=1}^m I_{k_i}(f_i)$. Then, for every $q=1,\dots,K=k_1+\cdots +k_m$, and for $\mu^q$-almost every $(z_1,\dots,z_q)\in \mathcal{Z}^q$ one has that
\begin{eqnarray}\label{e:kernels}
&&\frac{1}{q!} \mathbb{E}\left[D^{(q)}_{z_1,\dots,z_q} \Phi\right] \\ \label{e:zonta}
&&=\frac{1}{q!}  \sum_{W\in {\bf W}(q\, ;\, k_1,\dots,k_m)} \prod_{i=1}^m (k_i)_{(d_i)} \times \\ \notag
&& \quad \quad \times\!\! \sum_{\sigma\in\Pi_{\geq 2}(k_1-d_1,\dots,k_m-d_m)} \int_{\mathcal{Z}^{|\sigma|}} \left(f_1( z_{{\bf q}(1)}, \cdot)\otimes \cdots \otimes f_m( z_{{\bf q}(m)}\, , \cdot)\right)_\sigma\, d\mu^{|\sigma|}\\
&&  = h_q(z_1,..,z_q),\label{e:zappa}
\end{eqnarray}
where we have adopted the same notational conventions as in formula \eqref{e:prediagram}, and $h_q$ is the kernel appearing in \eqref{e:hq} ($q=1,...,K$).    
\end{proposition}
\begin{proof} We only need to prove identity \eqref{e:zappa}, that is, we have to show that, for a fixed $q=1,..,K-1$, the double sum appearing in \eqref{e:zonta} (without the prefactor $1/q!$) equals the sum on the right-hand side of \eqref{e:hq} (without the prefactor $1/q!$); the two sums are denoted by $\mathbb{S}_q(1)$ and $\mathbb{S}_q(2)$, respectively, in this proof. To prove the desired identity, we consider the partition $\pi^\star$ defined in \eqref{e:blocks}, whose blocks are denoted by $b^\star_1, ...,b^\star_m$, and introduce some ad-hoc notation. We write $\Theta(k_1,..,k_m)$ to denote the collection of ordered $q$-plets $T =(T_1,...,T_q)$ of {\it disjoint} subsets of $[K]$ ($K=k_1+\cdots + k_m$) such that each $T_\ell$, $\ell = 1,...,q$, contains at most one element of each block of $\pi^*$. For every $T= (T_1,...,T_q)\in \Theta(k_1,..,k_m)$ and every $i=1,...,m$, we write
$$
b_i^\star(T) := b_i^\star \, \backslash \, \cup_{\ell=1}^q T_\ell,
$$
that is, $b_i^\star(T)$ is the collection of those elements of $b_i^\star$ that are not contained in any coordinate of $T$; write $B^\star(T):= \cup_{i=1}^m b^\star_i(T)\subseteq [K]$. Finally, given $T\in \Theta(k_1,..,k_m)$ we use the symbol $\Pi_{\geq 2}(B(T))$ to indicate the set of all partitions $\varrho$ of $B^\star(T)$ such that every block of $\varrho$ has at least two elements, and every block of $\varrho$ has at most one element in common with each set $b_i^\star(T)$, $i=1,...,m$. Given $T\in \Theta(k_1,..,k_m)$ and $\varrho = (r_1,...,r_s) \in \Pi_{\geq 2}(B(T))$ (the $r_i$'s are the blocks of $\varrho$), we define the function 
$$
(z_1,...,z_q)\mapsto (f_1\otimes\cdots \otimes f_m)_{T,\varrho}(z_1,...z_q)
$$
as follows:
\begin{itemize} 
\item[--] Consider the (tensor product) function $f_1\otimes\cdots \otimes f_m$, as defined in \eqref{e:tensor}, and use an arbitrary vector {$(v_1,...,v_K)$} as its argument;
\item[--] For $\ell=1,...,q$, replace every variable {$v_j$} such that $j\in T_\ell$ with the variable $z_\ell$;
\item[--] For every $k=1,...,s$, replace every coordinate {$v_i$} such that $i\in r_k$ ($r_k$ is the $k$th block of $\varrho$) with a common variable $u_k$;
\item[--] Integrate the vector $(u_1,...,u_s)$ with respect to the product measure $\mu^s$ on $\mathcal{Z}^s$.
\end{itemize}
Then, a direct inspection shows that, with the notation introduced at the beginning of the present proof,
$$
\mathbb{S}_q(1)  = \sum_{\substack{T\in \Theta(k_1,..,k_m)\\ \varrho\in  \Pi_{\geq 2}(B(T)) }}(f_1\otimes\cdots \otimes f_m)_{T,\varrho}\, (z_1,...z_q) = \mathbb{S}_q(2),
$$
for $d\mu^q$-a.e. $(z_1,...,z_q)$. The proof is concluded.

\end{proof}

\medskip 

Plainly, in the case $q=K$, formula \eqref{e:kernels} coincides with \eqref{e:symtens}.

\medskip 

The final section of the paper is devoted to the proof of a new $p$-integrabilty criterion for products of random variables having a finite chaos expansion. This is the missing item to conclude the proof of Theorem \ref{t:mainintro}.

\section{A general criterion (and proof of Theorem \ref{t:mainintro})}\label{sec:MainResu}

\subsection{A general statement}

The following result yields necessary and sufficient conditions, implying that the product of random variables living in a finite sum of Wiener chaoses. It is one of the main achievements of the paper.

\medskip 
    
\begin{theorem}\label{t:geo}Fix $m\geq 1$ and consider integers $k_1,\dots,k_m\geq 1$. Let $F_1,\dots,F_m \in L^2(\Prob)$ be such that
$$
F_i \in \bigoplus_{q = 0}^{k_i} C_q, \quad i=1,\dots,m,
$$
and set $K := k_1+\cdots +k_m$ and $\Phi := \prod_{i=1}^m F_i$. For $p\in [1,2]$, consider the following conditions:
\begin{enumerate}
    \item[\bf (i-$p$)] $\Phi\in L^p(\Prob)$;
    \item[{\bf (ii-$p$)}] for all $q=1,\dots,K$, one has that
    $$
    \sum_{|W| = q } Q(D_{z_1,\dots,z_q}^{[W]}(F_1,\dots,F_m))\in L^1(\Prob), \quad \mbox{for $\mu^q$-a.e. $z_1,\dots,z_q$},
    $$
    and the mapping
    $$
(z_1,\dots,z_q)\mapsto \mathbb{E}\left[\sum_{|W| = q } Q(D_{z_1,\dots,z_q}^{[W]}(F_1,\dots,F_m))\right]
    $$
    is in $L^p(\mu^q)$.
\end{enumerate}
Then, one has that {\bf (ii-$p$)} implies {\bf (i-$p$)} for all $p\in [1,2)$, and also that {\bf (i-$2$)} and {\bf (ii-$2$)} are equivalent. Moreover, if either condition {\bf (i-$2$)} or {\bf (ii-$2$)} is satisfied, then the chaotic decomposition \eqref{chaosdec} of $\Phi$ is such that $h_q = 0$ for all $q>K$, and
\begin{equation}\label{e:finstroock}
h_q(z_1,\dots,z_q) = \frac{1}{q!} \mathbb{E}\left[\sum_{|W| = q } Q(D_{z_1,\dots,z_q}^{[W]}(F_1,\dots,F_m)) \right], \quad q=1,\dots,K.
\end{equation}
\end{theorem}
\begin{proof} The implication {\bf (i-$2$)} $ \to $ {\bf (ii-$2$)} directly follows from Theorem \ref{t:lastpenrose} and \eqref{e:dpprod}. The implication {\bf (ii-$p$)} $ \to $ {\bf (i-$p$)} for all $p\in [1,2]$ is a consequence of Lemma \ref{l:order}, Proposition \ref{p:doblerpeccati} and, again, formula \eqref{e:dpprod} (which also yields the final assertion in the statement). 
\end{proof}

\begin{remark}{\rm For $p\in [1,2)$, the implication {\bf (i-$p$) $\to$ (ii-$p$)} is false in general, even in the case $m=2$ and $k_1 = k_2 = 1$. To see this, consider the case were $\mu(\mathcal{Z}) = +\infty$, and select kernels $f_1, f_2 \in L^2(\mu)$ such that $f_1, f_2$ have disjoint supports (so that $I_1(f_1)$ and $I_1(f_2)$ are independent) and $f_1\notin L^p(\mu)$ for $p<2$. Then $K=2$, $h_1$ in formula \eqref{e:finstroock} equals zero, and $h_2 = {\rm sym}(f_1\otimes f_2) \notin L^p(\mu^2)$ for all $p\in [1,2)$. On the other hand, one has that $\Phi = I_1(f_1)\cdot I_1(f_2)$ is in $L^2(\Prob)$, and consequently in $L^p(\Prob)$ for all $p\in [1,2)$.}
\end{remark}

\medskip 

The implication {\bf (i-$2$)} $ \to $ {\bf (ii-$2$)} provides necessary conditions for the square-integrability of the random variable $\Phi := \prod_{i=1}^m F_i$. When applied to vectors of multiple integrals, such a result can be combined with the content of Example \ref{ex:wick}-(2) to deduce the following statement.

\medskip 

\begin{corollary}\label{c:wick} Let $k_1,\dots,k_m\geq 1$ be integers, and let the kernels $f_i \in L^2(\mu^{k_i})$, $i=1,\dots,m$, be such that the product $\Phi := \prod_{i=1}^m I_{k_i}(f_i)$ is square-integrable. Then, writing $K=k_1+\cdots +k_m$, one has that $\Phi \in \sum_{q=0}^K C_q$ and the projection of $\Phi$ on $C_K$ coincides with the multiple integral
\begin{equation}\label{e:wick}
I_K(f_1\otimes \cdots \otimes f_m),
\end{equation}
where we have used the notation \eqref{e:tensor}.
\end{corollary}

\medskip 

\begin{remark}\label{rem:AnswerProblem2}{\rm The fact that the multiple integral in \eqref{e:wick} is the projection of $\Phi$ on $C_K$ can be succinctly rewritten by using the language of {\bf Wick calculus} (see e.g. \cite[formula (1.6)]{surgailis84}), as follows:
\begin{equation}\label{e:simpletensorg}
: I_{k_1}(f_1)\cdots I_{k_m}(f_m) :\,\, = I_K(f_1\otimes\cdots \otimes f_m),
\end{equation}
where the left-hand side of the previous equation indicates a Wick product. The fact that the square-integrability of $\Phi = \prod_{i=1}^m I_{k_i}(f_i)$ implies that $\Phi \in \bigoplus_{p=0}^K C_p$ and that \eqref{e:simpletensorg} holds was conjectured in \cite[p. 222]{surgailis84} in the case $k_1=\cdots = k_m=1$ (and, to the best of our knowledge, never explicitly proved since).
}
\end{remark}

\subsection{End of the proof of Theorem \ref{t:mainintro}}

{\bf Part I} is a direct consequence of Theorem \ref{t:geo} in the case $p=2$, combined with Corollary \ref{c:wick}. {\bf Part II} follows directly from Proposition \ref{prop: OneCostCondiA}, that one has to use in synergy with Remark \ref{r:alocc}.

\subsection*{Acknowledgements} Research supported by the Luxembourg National Research Fund (Grant:
021/16236290/HDSA). The authors are grateful to Tara Trauthwein for useful discussions.

\bibliography{biblio}

\begin{thebibliography}{10}

\bibitem{AdamczakPivovarovSimanjuntak2024}
R.~Adamczak, P.~Pivovarov, and P.~Simanjuntak.
\newblock Limit theorems for the volumes of small codimensional random sections
  of {$\ell^n_p$}-balls.
\newblock {\em Ann. Probab.}, 52(1):93--126, 2024.

\bibitem{AkinwandeReitzner2020}
G.~Akinwande and M.~Reitzner.
\newblock Multivariate central limit theorems for random simplicial complexes.
\newblock {\em Adv. in Appl. Math.}, 121:102076, 27, 2020.

\bibitem{BachmannPeccati2016}
S.~Bachmann and G.~Peccati.
\newblock Concentration bounds for geometric {P}oisson functionals: logarithmic
  {S}obolev inequalities revisited.
\newblock {\em Electron. J. Probab.}, 21:Paper No. 6, 44, 2016.

\bibitem{BachmannReitzner2018}
S.~Bachmann and M.~Reitzner.
\newblock Concentration for {P}oisson {$U$}-statistics: subgraph counts in
  random geometric graphs.
\newblock {\em Stoch. Process. Appl.}, 128(10):3327--3352, 2018.

\bibitem{BaciBonnetThale2022}
A.~Baci, G.~Bonnet, and C.~Th\"ale.
\newblock Weak convergence of the intersection point process of {P}oisson
  hyperplanes.
\newblock {\em Ann. Inst. Henri Poincar\'e{} Probab. Stat.}, 58(2):1208--1227,
  2022.

\bibitem{BetkenHugThale2023}
C.~Betken, D.~Hug, and C.~Th\"ale.
\newblock Intersections of {P}oisson {$k$}-flats in constant curvature spaces.
\newblock {\em Stoch. Process. Appl.}, 165:96--129, 2023.

\bibitem{BPY}
C.~Bhattacharjee, G.~Peccati, and D.~Yogeshwaran.
\newblock Spectra of {Poisson} functionals and applications in continuum
  percolation.
\newblock Preprint, {arXiv}:2407.13502, 2024.

\bibitem{bobkovledoux}
S.~G. Bobkov and M.~Ledoux.
\newblock On modified logarithmic {Sobolev} inequalities for {Bernoulli} and
  {Poisson} measures.
\newblock {\em J. Funct. Anal.}, 156(2):347--365, 1998.

\bibitem{BourguinCampeseDang2024}
S.~Bourguin, S.~Campese, and T.~Dang.
\newblock Functional {G}aussian approximations on {H}ilbert-{P}oisson spaces.
\newblock {\em ALEA Lat. Am. J. Probab. Math. Stat.}, 21(1):517--553, 2024.

\bibitem{BourguinDurastanti2017}
S.~Bourguin and C.~Durastanti.
\newblock On high-frequency limits of {$U$}-statistics in {B}esov spaces over
  compact manifolds.
\newblock {\em Illinois J. Math.}, 61(1-2):97--125, 2017.

\bibitem{BourguinPeccati2014}
S.~Bourguin and G.~Peccati.
\newblock Portmanteau inequalities on the {P}oisson space: mixed regimes and
  multidimensional clustering.
\newblock {\em Electron. J. Probab.}, 19:no. 66, 42, 2014.

\bibitem{chafai8}
D.~Chafa{\"{\i}}.
\newblock Entropies, convexity, and functional inequalities: on
  {{\(\Phi\)}}-entropies and {{\(\Phi\)}}-{Sobolev} inequalities.
\newblock {\em J. Math. Kyoto Univ.}, 44(2):325--363, 2004.

\bibitem{CongXia2024}
T.~Cong and A.~Xia.
\newblock Normal approximation in total variation for statistics in geometric
  probability.
\newblock {\em Adv. in Appl. Probab.}, 56(1):106--155, 2024.

\bibitem{DecHomology}
L.~Decreusefond, E.~Ferraz, H.~Randriambololona, and A.~Vergne.
\newblock Simplicial homology of random configurations.
\newblock {\em Adv. in Appl. Probab.}, 46(2):325--347, 2014.

\bibitem{DTGS23}
P.~Di~Tella, C.~Geiss, and A.~Steinicke.
\newblock Product formulas for multiple stochastic integrals associated with
  {L}{\'e}vy processes.
\newblock {\em Collect. Math.}, pages 1--33, 2024.

\bibitem{DoblerPeccati2017}
C.~D\"obler and G.~Peccati.
\newblock Quantitative de {J}ong theorems in any dimension.
\newblock {\em Electron. J. Probab.}, 22:Paper No. 2, 35, 2017.

\bibitem{DoblerPeccatiAOP}
C.~D{\"o}bler and G.~Peccati.
\newblock The fourth moment theorem on the {Poisson} space.
\newblock {\em Ann. Probab.}, 46(4):1878--1916, 2018.

\bibitem{DoblerPeccati18}
C.~D{\"o}bler and G.~Peccati.
\newblock Fourth moment theorems on the {Poisson} space: analytic statements
  via product formulae.
\newblock {\em Electron. Commun. Probab.}, 2018.

\bibitem{DP18a}
C.~D{\"o}bler and G.~Peccati.
\newblock Quantitative clts for symmetric {U}-statistics using contractions.
\newblock {\em Electron. J. Probab.}, 2019.

\bibitem{DVZ18}
C.~D\"{o}bler, A.~Vidotto, and G.~Zheng.
\newblock Fourth moment theorems on the {P}oisson space in any dimension.
\newblock {\em Electron. J. Probab.}, 23:1--27, 2018.

\bibitem{DurMaPec}
C.~Durastanti, D.~Marinucci, and G.~Peccati.
\newblock Normal approximations for wavelet coefficients on spherical {P}oisson
  fields.
\newblock {\em J. Math. Anal. Appl.}, 409(1):212--227, 2014.

\bibitem{EichelsbacherThale2014}
P.~Eichelsbacher and C.~Th\"ale.
\newblock New {B}erry-{E}sseen bounds for non-linear functionals of {P}oisson
  random measures.
\newblock {\em Electron. J. Probab.}, 19:no. 102, 25, 2014.

\bibitem{FisslerThale2016}
T.~Fissler and C.~Th\"ale.
\newblock A four moments theorem for gamma limits on a {P}oisson chaos.
\newblock {\em ALEA Lat. Am. J. Probab. Math. Stat.}, 13(1):163--192, 2016.

\bibitem{GieLast}
F.~Gieringer and G.~Last.
\newblock Concentration inequalities for measures of a {Boolean} model.
\newblock {\em ALEA, Lat. Am. J. Probab. Math. Stat.}, 15(1):151--166, 2018.

\bibitem{Grygierek2020}
J.~Grygierek.
\newblock Poisson and {G}aussian fluctuations for the components of the {$\bold
  {f}$}-vector of high-dimensional random simplicial complexes.
\newblock {\em ALEA Lat. Am. J. Probab. Math. Stat.}, 17(2):675--709, 2020.

\bibitem{GuSaTh}
A.~Gusakova, H.~Sambale, and Ch. Th{\"a}le.
\newblock Concentration on {Poisson} spaces via modified {{\(\Phi\)}}-{Sobolev}
  inequalities.
\newblock {\em Stochastic Processes Appl.}, 140:216--235, 2021.

\bibitem{Herry2020}
R.~Herry.
\newblock Stable limit theorems on the {P}oisson space.
\newblock {\em Electron. J. Probab.}, 25:Paper No. 149, 30, 2020.

\bibitem{HugLastSchulte2016}
D.~Hug, G.~Last, and M.~Schulte.
\newblock Second-order properties and central limit theorems for geometric
  functionals of {B}oolean models.
\newblock {\em Ann. Appl. Probab.}, 26(1):73--135, 2016.

\bibitem{HugThaleWeil2015}
D.~Hug, C.~Th\"ale, and W.~Weil.
\newblock Intersection and proximity of processes of flats.
\newblock {\em J. Math. Anal. Appl.}, 426(1):1--42, 2015.

\bibitem{Kab76}
Y.~M. Kabanov.
\newblock On extended stochastic intervals.
\newblock {\em Theory Probab. Appl.}, 1976.

\bibitem{KabluchkoRosenThale2024}
Z.~Kabluchko, D.~Rosen, and C.~Th\"ale.
\newblock A quantitative central limit theorem for {P}oisson horospheres in
  high dimensions.
\newblock {\em Electron. Commun. Probab.}, 29:Paper No. 47, 11, 2024.

\bibitem{LRPEJP}
R.~Lachi\`eze-Rey and G.~Peccati.
\newblock Fine {G}aussian fluctuations on the {P}oisson space, {I}:
  contractions, cumulants and geometric random graphs.
\newblock {\em Electron. J. Probab.}, 18:no. 32, 32, 2013.

\bibitem{LRPSPA}
R.~Lachi\`eze-Rey and G.~Peccati.
\newblock Fine {G}aussian fluctuations on the {P}oisson space {II}: rescaled
  kernels, marked processes and geometric {$U$}-statistics.
\newblock {\em Stoch. Process. Appl.}, 123(12):4186--4218, 2013.

\bibitem{LRRsv}
R.~Lachi\`eze-Rey and M.~Reitzner.
\newblock {$U$}-statistics in stochastic geometry.
\newblock In {\em Stochastic analysis for {P}oisson point processes}, volume~7
  of {\em Bocconi Springer Ser.}, pages 229--253. Bocconi Univ. Press, [place
  of publication not identified], 2016.

\bibitem{Lastsv}
G.~Last.
\newblock Stochastic analysis for {P}oisson processes.
\newblock In G.~Peccati and M.~Reitzner, editors, {\em Stochastic analysis for
  Poisson point processes}, Mathematics, Statistics, Finance and Economics,
  chapter~1, pages 1--36. Bocconi University Press and Springer, 2016.

\bibitem{LPS}
G.~Last, G.~Peccati, and M.~Schulte.
\newblock Normal approximation on {P}oisson spaces: {M}ehler's formula, second
  order {P}oincar\'e inequalities and stabilization.
\newblock {\em Probab.Theory Relat. Fields}, 2016.

\bibitem{LPY}
G.~Last, G.~Peccati, and D.~Yogeshwaran.
\newblock Phase transitions and noise sensitivity on the {P}oisson space via
  stopping sets and decision trees.
\newblock {\em Random Struct. Algorithms}, 63(2):457--511, 2023.

\bibitem{LPfock}
G.~Last and M.~Penrose.
\newblock Poisson process {F}ock space representation, chaos expansion and
  covariance inequalities.
\newblock {\em Probab. Theory Relat. Fields (}, 2011.

\bibitem{LPbook}
G.~Last and M.~Penrose.
\newblock {\em Lectures on the Poisson Process}.
\newblock IMS Textbooks. Cambridge University Press, Cambridge, 2017.

\bibitem{LaPen11}
G.~Last and M.~D. Penrose.
\newblock Poisson process {F}ock space representation, chaos expansion and
  covariance inequalities.
\newblock {\em Probab.Theory Relat. Fields}, 150(3-4):663--690, 2011.

\bibitem{LPST}
G.~Last, M.~D. Penrose, M.~Schulte, and C.~Th{\"a}le.
\newblock Moments and central limit theorems for some multivariate {Poisson}
  functionals.
\newblock {\em Adv. Appl. Probab.}, 46(2):348--364, 2014.

\bibitem{LeMinh2023}
T.~Le~Minh.
\newblock {$U$}-statistics on bipartite exchangeable networks.
\newblock {\em ESAIM Probab. Stat.}, 27:576--620, 2023.

\bibitem{NouPecbook}
I.~Nourdin and G.~Peccati.
\newblock {\em Normal approximations with {M}alliavin calculus}, volume 192 of
  {\em Cambridge Tracts in Mathematics}.
\newblock Cambridge University Press, Cambridge, 2012.
\newblock From Stein's method to universality.

\bibitem{NPYpams}
I.~Nourdin, G.~Peccati, and X.~Yang.
\newblock Restricted hypercontractivity on the {P}oisson space.
\newblock {\em Proc. Amer. Math. Soc.}, 148(8):3617--3632, 2020.

\bibitem{PRbook}
G.~Peccati and M.~Reitzner, editors.
\newblock {\em Stochastic analysis for {P}oisson point processes}, volume~7 of
  {\em Bocconi \& Springer Series}.
\newblock Bocconi University Press; Springer, 2016.
\newblock Malliavin calculus, Wiener-It\^o{} chaos expansions and stochastic
  geometry.

\bibitem{PSTU}
G.~Peccati, J.~L. Sol{\'e}, M.~S. Taqqu, and F.~Utzet.
\newblock Stein's method and normal approximation of {Poisson} functionals.
\newblock {\em Ann. Probab.}, 38(2):443--478, 2010.

\bibitem{PTbook}
G.~Peccati and M.~S. Taqqu.
\newblock {\em Wiener chaos: moments, cumulants and diagrams}.
\newblock Bocconi \& Springer Series. Springer, Milan; Bocconi University
  Press, Milan, 2011.

\bibitem{PeccatiThaeleALEA}
G.~Peccati and C.~Th\"ale.
\newblock Gamma limits and {$U$}-statistics on the {P}oisson space.
\newblock {\em ALEA Lat. Am. J. Probab. Math. Stat.}, 10(1):525--560, 2013.

\bibitem{PianoforteTurin2023}
F.~Pianoforte and R.~Turin.
\newblock Multivariate {P}oisson and {P}oisson process approximations with
  applications to {B}ernoulli sums and {$U$}-statistics.
\newblock {\em J. Appl. Probab.}, 60(1):223--240, 2023.

\bibitem{privaultbook}
N.~Privault.
\newblock {\em Stochastic analysis in discrete and continuous settings with
  normal martingales}.
\newblock Lecture Notes in Mathematics. Springer-Verlag, Berlin, 2009.

\bibitem{SchulteThaele12}
M.~Reitzner and M.~Schulte.
\newblock Central limit theorems for {$U$}-statistics of {P}oisson point
  processes.
\newblock {\em Ann. Probab.}, 41(6):3879--3909, 2013.

\bibitem{ReitznerSchulteThale2017}
M.~Reitzner, M.~Schulte, and C.~Th\"ale.
\newblock Limit theory for the {G}ilbert graph.
\newblock {\em Adv. in Appl. Math.}, 88:26--61, 2017.

\bibitem{RW97}
G.-C. Rota and T.~C. Wallstrom.
\newblock Stochastic integrals: a combinatorial approach.
\newblock {\em Ann. Probab.}, 1997.

\bibitem{STT24}
H.~Sambale, C.~Thäle, and T.~Trauthwein.
\newblock Central limit theorems for the nearest neighbour embracing graph in
  {E}uclidean and hyperbolic space.
\newblock {\em Stoch. Process. Appl.}, 188:104671, 2025.

\bibitem{Schulte12}
M.~Schulte.
\newblock A central limit theorem for the {P}oisson-{V}oronoi approximation.
\newblock {\em Adv. in Appl. Math.}, 49(3-5):285--306, 2012.

\bibitem{Schulte2016}
M.~Schulte.
\newblock Normal approximation of {P}oisson functionals in {K}olmogorov
  distance.
\newblock {\em J. Theoret. Probab.}, 29(1):96--117, 2016.

\bibitem{ST24}
M.~Schulte and C.~Th{\"a}le.
\newblock Moderate deviations on {P}oisson chaos.
\newblock {\em Electron. J. Probab.}, 29:1--27, 2024.

\bibitem{SchulteThale2024}
M.~Schulte and C.~Th\"ale.
\newblock Moderate deviations on {P}oisson chaos.
\newblock {\em Electron. J. Probab.}, 29:Paper No. 146, 27, 2024.

\bibitem{surgailis84}
D.~Surgailis.
\newblock On multiple {Poisson} stochastic integrals and associated {Markov}
  semigroups.
\newblock {\em Probab. Math. Stat.}, 3(2):217--239, 1984.

\bibitem{Thomas2023}
A.~M. Thomas.
\newblock Central limit theorems and asymptotic independence for local
  {$U$}-statistics on diverging halfspaces.
\newblock {\em Bernoulli}, 29(4):3280--3306, 2023.

\bibitem{multitara}
T.~Trauthwein.
\newblock Multivariate second-order $p$-{P}oincar\'e inequalities.
\newblock Preprint, {arXiv}:2409.02843, 2024.

\bibitem{trauthwein}
T.~Trauthwein.
\newblock Quantitative {CLT}s on the {P}oisson space via {S}korohod estimates
  and $p$-{P}oincaré inequalities.
\newblock {\em Ann. Appl. Probab.}, 2025.

\bibitem{wu2000}
L.~Wu.
\newblock A new modified logarithmic {Sobolev} inequality for {Poisson} point
  processes and several applications.
\newblock {\em Probab. Theory Relat. Fields}, 118(3):427--438, 2000.

\end{thebibliography}
\bibliographystyle{plain}

\end{document}